\documentclass[11pt]{article}
\usepackage{epigamath}

\usepackage[notext]{kpfonts}
\usepackage{baskervald}

\setpapertype{A4}

\usepackage[english]{babel}

\usepackage[dvips]{graphicx}     

\title{\vspace{-0.5cm}Crepant Resolutions and Open Strings II}
\titlemark{Crepant Resolutions and Open Strings II}
\author{\vspace{0cm} Andrea Brini and Renzo Cavalieri}
\authoraddresses{
\authordata{Andrea Brini}{\firstname{Andrea} \lastname{Brini}\\
\institution{IMAG, Univ. Montpellier, CNRS, Montpellier, France
\newline  Department of Mathematics, Imperial College, 180 Queen's Gate, London SW7
  2AZ, United Kingdom}\\
\email{a.brini@imperial.ac.uk}}\\
\authordata{Renzo Cavalieri}{\firstname{Renzo}
\lastname{Cavalieri}\\
\institution{Department of Mathematics, Colorado State University, 101 Weber Building, Fort
Collins, CO 80523-1874, USA}\\
\email{renzo@math.colostate.edu}}
}
\authormark{A. Brini and R. Cavalieri}
\date{\vspace{-5ex}} 
\journal{\'Epijournal de G\'eom\'etrie Alg\'ebrique} 
\acceptation{Received by the Editors on August 25, 2017, and in final form
on February 28, 2018. \\ Accepted on May 12, 2018.}

\newcommand{\articlereference}{Volume 2 (2018), Article Nr. 4}

 \usepackage[all]{xy}

\allowdisplaybreaks

\newtheorem{theorem}{Theorem}[section]

\newtheorem{example}[theorem]{Example}

\newcommand{\X}{\mathcal{X}}
\newcommand{\Z}{\mathbb{Z}}
\newcommand{\C}{\mathbb{C}}

 \DeclareMathOperator{\diag}{diag}
\newcommand{\de}{{\partial}}

\newcommand{\rd}{\mathrm{d}}
\newcommand{\ri}{\mathrm{i}}
\newcommand{\re}{\mathrm{e}}

\newcommand{\U}{\mathbb{U}_{\rho}^{\X,Y}}

\newcommand{\bbN}{\mathbb{N}}

\newcommand{\bbZ}{\mathbb{Z}}
\newcommand{\bbR}{\mathbb{R}}
\newcommand{\bbC}{\mathbb{C}}
\newcommand{\bbO}{\mathbb{O}}
\newcommand{\bbP}{\mathbb{P}}
\newcommand{\bbF}{\mathbb{F}}

\newcommand{\bbV}{\mathbb{V}}

\def\bary{\begin{array}}
\def\eary{\end{array}}
\def\ben{\begin{enumerate}}
\def\een{\end{enumerate}}
\def\bit{\begin{itemize}}
\def\eit{\end{itemize}}
\def\nn{\nonumber}

\newcommand{\cY}{\mathcal{Y}}

\newcommand{\cO}{\mathcal{O}}

\newcommand{\cC}{\mathcal{C}}
\newcommand{\DD}{\mathcal{D}}
\newcommand{\LL}{\mathcal{L}}
\newcommand{\cS}{\mathcal{S}}

\newcommand{\cN}{\mathcal{N}}
\newcommand{\cW}{\mathcal{W}}

\newcommand{\cA}{\mathcal{A}}
\newcommand{\HH}{\mathcal{H}}
\newcommand{\cB}{\mathcal{B}}
\newcommand{\cF}{\mathcal{F}}
\newcommand{\cI}{\mathcal{I}}

\newcommand{\RR}{\mathcal{R}}
\newcommand{\cX}{\mathcal{X}}

\newcommand{\cM}{\mathcal M}

\newcommand{\bfz}{\mathbf{z}}

\newcommand{\eq}[1]{\begin{equation}#1\end{equation}}
\newcommand{\ea}[1]{\begin{align}#1\end{align}}
\def\bary{\begin{array}}
\def\eary{\end{array}}
\def\ben{\begin{enumerate}}
\def\een{\end{enumerate}}
\def\bit{\begin{itemize}}
\def\eit{\end{itemize}}
\def\nn{\nonumber}
\def\de {\partial}

\def\a{\alpha}
\def\b{\beta}

\newtheorem{thm}{Theorem}[section]
\newtheorem{lem}[thm]{Lemma}

\newtheorem{prop}[thm]{Proposition}
\newtheorem{conj}[thm]{Conjecture}

\newtheorem{proposal}{Proposal}
\newtheorem{cor}[thm]{Corollary}

\newtheorem{defn}{Definition}[section]

\newcommand{\cref}[1]{{\rm (\ref{#1})}}

\newtheorem{rmk}[thm]{Remark}

\newcommand{\Li}{\operatorname{Li}}

\newcommand{\GIT}[1]{/\!\!/_{\kern-.2em #1 \kern0.1em}}

\renewcommand{\Re}{\mathfrak{Re}}

\renewcommand{\l}{\left}
\renewcommand{\r}{\right}
\newcommand{\bra}{\left\langle}
\newcommand{\ket}{\right\rangle}

\newcommand{\ev}{\operatorname{ev}}
\usepackage{color}

\def\bes{\begin{subequations}}
\def\ees{\end{subequations}}

\begin{document}


\maketitle



\begin{prelims}

\vspace{-0.55cm}

\def\abstractname{Abstract}
\abstract{We recently formulated a number of  Crepant Resolution
Conjectures (CRC) for open
Gromov--Witten invariants of Aganagic--Vafa Lagrangian branes and verified them for the
family of threefold type $A$-singularities. In this paper we enlarge the body
of evidence in favor of our open CRCs, along two
different strands. In one direction, we  consider non-hard Lefschetz targets
and verify the disk CRC for local weighted projective planes. In the
other, we complete the proof of the quantum (all-genus) open CRC for hard Lefschetz toric
Calabi--Yau three dimensional representations by a detailed study of the $G$-Hilb resolution of
$[\mathbb{C}^3/G]$ for $G=\mathbb{Z}_2 \times \mathbb{Z}_2$. Our results have implications for
closed-string CRCs of Coates--Iritani--Tseng, Iritani, and Ruan
for this class of examples.}

\keywords{Crepant resolution conjecture; Gromov-Witten theory; open invariants; quantum cohomology; orbifold cohomology; mirror symmetry}

\MSCclass{14N35; 53D45}

\vspace{0.05cm}

\languagesection{Fran\c{c}ais}{%

\textbf{Titre. R\'esolutions cr\'epantes et cordes ouvertes II} \commentskip
\textbf{R\'esum\'e.} Nous avons r\'ecemment formul\'e un ensemble de
Conjectures de R\'esolutions Cr\'epantes (CRC) pour les invariants de Gromov--Witten ouverts des branes lagrangiennes de Aganagic--Vafa, et nous les avons
v\'erifi\'ees pour la famille des singularit\'es transverses de type $A$ en
dimension trois. Dans cet article, nous \'elargissons le faisceau de preuves en faveur de nos CRC ouvertes, et ce dans deux directions.
Dans la premi\`ere, nous consid\'erons des cibles  satisfiant la condition
dite de ``Lefschetz forte'' et v\'erifions la CRC du disque pour des plans projectifs \`a poids locaux.
Dans l'autre, nous compl\'etons la d\'emonstration de toutes les CRC ouvertes
quantiques (en tout genre) pour les repr\'esentations tridimensionnelles
toriques de type Calabi--Yau et v\'erifiant la condition de Lefschetz forte, ceci se faisant \`a travers une \'etude d\'etaill\'ee de la r\'esolution $G$-Hilb de
$[\mathbb{C}^3/G]$ pour $G=\mathbb{Z}_2 \times \mathbb{Z}_2$. Nos r\'esultats ont des cons\'equences sur les CRC pour les cordes ferm\'ees de Coates--Iritani--Tseng, Iritani et Ruan
pour cette classe d'exemples.}

\end{prelims}


\newpage

\setcounter{tocdepth}{1} \tableofcontents

\section{Introduction}

In a recent paper \cite{Brini:2013zsa}, we proposed two versions of
a Crepant Resolution Conjecture for open Gromov--Witten invariants
of Aganagic--Vafa orbi-branes inside semi-projective toric
Calabi--Yau 3-orbifolds:

\begin{itemize}
\item a general Bryan--Graber-type comparison between disk potentials after analytic
continuation ({\bf the disk CRC});
\item a stronger identification of the full open string partition function
at all genera and arbitrary boundary components for hard Lefschetz
targets ({\bf the quantized open CRC}).
\end{itemize}

We recall these statements more precisely in Section
\ref{sec:review}. Both conjectures were proved in
\cite{Brini:2013zsa} for the case of the crepant resolutions of
type~A threefold singularities, but they are expected to hold in
wider generality.
In particular, the disk CRC should hold true for general (non-hard
Lefschetz) toric CY3 that are projective over their affinization;
moreover, the proof of the quantized open CRC in
\cite{Brini:2013zsa} left out one exceptional example of (toric)
hard Lefschetz crepant resolution. The purpose of this paper is to
offer further evidence of the general validity of the disk CRC, as
well as to conclude the proof of the quantized open CRC for hard
Lefschetz toric three dimensional representations. \\

The first problem we  tackle is the disk CRC for non-hard Lefschetz
targets. We  concentrate our attention to local weighted projective
planes: our poster-child is the partial crepant resolution
$\pi:K_{\bbP(1,1,n)}\to \bbC^3/\bbZ_{n+2}$, where $\pi$ contracts
the image of the zero section to give the quotient singularity
$\frac{1}{n+2}(1,1,-2)$. In
particular, we establish the following \\

\noindent{\bf Theorem 1} [(Theorem \ref{thm:crcproj} and Corollary \ref{cor:crcproj})]: the
disk CRC holds for $Y=K_{\bbP(n,1,1)}$
and $\cX=[\bbC^3/\bbZ_{n+2}]$.\\

On a somewhat orthogonal direction, we complete the study of hard
Lefschetz crepant resolutions of three dimensional representations
by considering the $G$-Hilb resolution of $[\bbC^3/G]$ for $G=\bbZ_2
\times \bbZ_2$ -- the so-called {\it closed topological vertex}
geometry
studied in \cite{MR2129009}. \\


\noindent{\bf Theorem 2} [(Theorem \ref{thm:scr} and Corollary \ref{cor:scr})]: the quantized
CRC holds
for $\cX=[\bbC^3/\bbZ_{2}\times \bbZ_2]$ and $Y$ its canonical $G$-Hilb resolution.\\

In \cite{cavalieri2011open}, it was shown in detail in the specific
example of the $A_1$ threefold singularity that the local CRC for
$[\bbC^3/\bbZ_2]$ glues to a crepant resolution statement for
$K_{\bbP^1\times \bbP^1} \to [\cO(-1)_{\bbP^1}\oplus
\cO(-1)_{\bbP^1}/\bbZ_2]$. Theorem~2, the results of
\cite{Brini:2013zsa}, and a suitable generalization of the gluing
theorem of \cite{cavalieri2011open} would together imply the all
genus open CRC for all toric hard Lefschetz CY3 targets.

\subsection*{Context and further discussion}

Good part of the proof of Theorem~1
 relies on the well-established mirror symmetry framework of
 \cite{MR2529944,MR2486673}: we construct twisted
$I$-functions as hypergeometric modifications of the untwisted ones
 and then study their
analytic continuation corresponding to a change of chamber in the
K\"ahler moduli space of the target. The first step is standard
\cite{MR2578300,MR2276766,MR2510741}; for the second, we overcome
the technical intricacies of the Mellin--Barnes method
\cite{MR2486673} through a combined use of hypergeometric
resummation and a generalized Kummer-type connection formula for the
analytic continuation across a single wall. This technique has a
number of features of independent interest: it turns out to be
significantly more powerful than the usual Mellin--Barnes method,
and it is applicable to the study of wall-crossings in toric
Gromov--Witten theory in quite large generality. In particular, it
might be applied in combination with the mirror theorem of
\cite{ccit2} for the study, and hopefully the proof, of the
closed-string CRC
in the toric setting. \\

As for Theorem 2, our strategy to prove it follows closely ideas of
\cite{Brini:2013zsa} for the case of $[\bbC^2/\bbZ_{n} \times
  \bbC]$. In \cite{Brini:2013zsa,Brini:2014mha},
the Gromov--Witten/Integrable Systems  was employed to exhibit a
one-dimensional Landau--Ginzburg mirror model for the equivariant
quantum cohomology of type~A resolutions: the relevant
superpotential was identified with the dispersionless Lax function
of the $q$-deformed $(n+1)$-KdV hierarchy. For the case of
$[\bbC^3/\bbZ_2 \times
  \bbZ_2]$, the relevant Frobenius manifold turns out to be the coefficient space of a particular reduction of the
genus-zero Whitham hierarchy with three marked points
\cite{Krichever:1992qe}; a detailed study of this system and its
bihamiltonian structure will appear elsewhere. As was the case in
\cite{Brini:2013zsa}, this has two main upshots: in genus zero, it
allows a one-step study of wall-crossing beyond multiple walls; and
in higher genus, it significantly reduces the complexity of the
proof of the quantized version of the open CRC, which turns into an
exercise in all-order classical Laplace
asymptotics. \\

Limited to the class of examples considered here, our results  also
have implications for ordinary (closed) Crepant Resolution
Conjectures of Iritani \cite{MR2553377} and
Coates--Iritani--Tseng/Ruan \cite{MR2529944,MR3112518}. The proof of
the disk CRC in Section \ref{sec:localpp} establishes in particular a
natural fully-equivariant version of Iritani's $K$-theoretic Crepant
Resolution Conjecture for the examples at hand\footnote{A much more
general proof for semi-projective toric orbifolds
  has been announced by Coates--Iritani--Jiang.},
whereas the study of the quantized OCRC in Section \ref{sec:ctv} leads us
to verify the all-genus closed CRC with descendents for
$\cX=[\bbC^3/\bbZ_2 \times \bbZ_2]$.

\subsection*{Plan of the paper} The paper is organized as follows. In
Section \ref{sec:review}, we concisely review our setup in
\cite{Brini:2013zsa} for the disk and the quantized open CRC. We
then furnish a proof of the disk CRC in Section \ref{sec:localpp}, and
study its implications at the level of scalar potentials for each of
the two brane setups allowed by the geometry. In Section \ref{sec:ctv} we
study the closed topological vertex geometry: we first present a
mirror description in terms of a one-dimensional logarithmic
Landau--Ginzburg model, which is then used in the analytic
continuation relevant for the disk CRC and the all-order asymptotic
analysis necessary to establish the quantized OCRC.

\subsection*{Acknowledgements}
The authors would like to thank Hiroshi Iritani, Douglas Ortego,
Stefano Romano, Dusty Ross and Mark Shoemaker for their discussions
and comments related to this project. The second author gratefully
acknowledges support by NSF grant DMS-1101549, NSF RTG grant
1159964.

\section{Crepant Resolution Conjectures: a review}
\label{sec:review}

Given $\cX$ a Gorenstein algebraic orbifold and $Y \to X$ a crepant
resolution of its coarse moduli space, Ruan  conjectured
\cite{MR2234886} that the small quantum cohomologies of $Y$ and
$\cX$ should be isomorphic after analytic continuation and a
suitable identification of the quantum parameters. More recently,
Coates--Iritani--Tseng shaped -- and generalized -- Ruan's original
Crepant Resolution Conjecture (CRC) into a comparison of Lagrangian
cones via a symplectic isomorphism $\U:\HH_\cX\to\HH_Y$ between the
Givental spaces of $\cX$ and $Y$ \cite{MR2529944}; here $\rho$
denotes a choice of analytic continuation path. Further, Iritani's
theory of integral structures \cite{MR2553377} makes a prediction
for $\U$ based exclusively on the classical geometry of the targets.
In this section we briefly summarize some of the recent extensions
of the Coates--Iritani--Tseng CRC that this work relates to, and
that are relevant for our formulation of the CRC for open
Gromov--Witten invariants. Background, motivation, and extensive
discussions of the setup presented here can be found in our previous
paper \cite[Sec.~2
  and App.~A]{Brini:2013zsa};
the reader who is not familiar with the closed string CRC and its
higher genus analogues is referred to the  survey papers
\cite{MR3112518, MR2683208}.

\subsection{The disk CRC}
\label{sec:dcrc}

In \cite{Brini:2013zsa}, the authors formulate an Open Crepant
Resolution Conjecture (OCRC) as a comparison diagram relating
geometric objects in the Givental spaces of the targets, following
the philosophy of \cite{MR2529944}.
Let $\cW$ be a three-dimensional CY toric orbifold,  $p$ a fixed
point such that a neighborhood is isomorphic to $[\C^3/G]$, with $G
\cong \Z_{n_1}\times \ldots \times \Z_{n_l}$. The local  group
action is defined by the character vectors $(\vec{m}^1,
\vec{m}^2,\vec{m}^3)$ and a Calabi--Yau 2-torus action $T\simeq
(\bbC^*)^2$ is specified by weights  $(w_1,w_2,w_3)\in
H^\bullet_T({\rm pt})$. Fix a Lagrangian boundary condition $L$
which we assume to be on the first coordinate axis in this local
chart.
Define ${n_{\rm eff}}=\mbox{lcm} \{ n_j/\gcd(m_j^1,n_j) \ | j=
1,\ldots,l \}$ to be the size of the effective part of the action
along the first coordinate axis.
  There exist a map from an orbi-disk mapping to the first coordinate axis with winding $d$ and twisting\footnote{Here twisting refers to the image of the center of the disk in the evaluation map to the inertia orbifold.} $\vec{k}$ if the compatibility condition
\eq{ \frac{d}{{n_{\rm eff}}}-\sum_{j=1}^l\frac{k_jm_j^1}{n_j}\in
\bbZ \label{compat} }
is satisfied. Via the Atiyah--Bott isomorphism, the Chen--Ruan
cohomology ring of $[\C^3/G]$ is naturally identified with a part of
$H^\bullet_T(\mathcal{W})$, with generators $\mathbf{1_{p,k}}$.
Denoting by $\mathbf{1_{p}^{k}}$ the Poincar\'e dual of
$\mathbf{1_{p,k}}$, we define the disk tensor at $p$ as:
\begin{equation}
\overline{\DD}_{\cW,p}^+(z;\vec{w}) \triangleq \frac{\pi }{w_1|G|
\sin\left(\pi\left(\left\langle \sum_{j=1}^l\frac{k_jm_j^3}{n_j}
\right\rangle -\frac{w_{3}}{z}\right)\right)} \frac{1}{
\overline{\Gamma}_\cW^k }\mathbf{1_{p}^{k}}\otimes
\mathbf{1_{p}^{k}}, \label{eq:dr}
\end{equation}
where $ \overline{\Gamma}_\cW^k$ is the $\mathbf{1_{p,k}}$
coefficient of Iritani's homogenized Gamma function
(\cite{Brini:2013zsa}, Eqn. (27)). The global {\it disk tensor} for
$\cW$ is then defined as the sum of the disk tensors at the points
adjacent to the Lagrangian  $L$ in the toric diagram of $\cW$. Note
that $z$ is thought of as the descendant parameter  and hence
$\overline{\DD}_{\cW}^+(z;\vec{w})$ is naturally a tensor on
$\mathcal{H}_{\cW}$, the Givental space of $\cW$.

The {\it winding neutral disk potential} is defined to be the
contraction of the $J$ function of $\cW$ with the disk tensor.
Lowering  indices in the $J$ function with the Poincar\'e pairing,
we can write this as the composition:
\begin{equation}
\mathbb{F}_{L}^{\rm disk}(\tau,z,\vec{w}) \triangleq
\overline{\DD}_\cW^+ \circ J^\cW\left(\tau,z;\vec{w}\right).
\end{equation}

The winding neutral disk potential is a section of Givental space
that contains information about disk invariants at all winding, in
the sense that disk invariants  of winding $d$ appear in the
specialization of $\mathbb{F}_{L}^{\rm disk}(t,z,\vec{w})$ at $z=
n_{\rm eff} w_1/d$, as coefficients in front of monomials where the
compatibility  condition $\eqref{compat}$ is satisfied. Rather then
performing the specialization of the  variable $z$ to construct a
generating function for open invariants, we formulate the OCRC  as a
comparison diagram of winding neutral disk potentials, i.e. a
comparison among sections of Givental space.

\begin{proposal}[The OCRC] \label{theocrc}
For $\cW$ either $\cX$ or $Y$, let $\Delta_\cW$ denote the free
module in the cohomology of $\cW$ over $H(BT)$ spanned by the
$T$-equivariant lifts of Chen--Ruan cohomology classes having
age-shifted degree at most two. There exists a
$\bbC((z^{-1}))$-linear map of Givental spaces $\bbO: \HH_\X \to
\HH_Y$ and analytic functions $\mathfrak{h}_\cW: \Delta_\cW \to
\bbC$
such that \eq{ \mathfrak{h}_Y^{1/z}{\mathbb{F}_{L,Y}^{\rm
disk}}\big|_{\Delta_Y}= \mathfrak{h}_\cX^{1/z} \bbO\circ
\mathbb{F}_{L,\X}^{\rm disk}\big|_{\Delta_\cX} } upon analytic
continuation of quantum cohomology parameters.
\end{proposal}

The analytic functions $\mathfrak{h}_\cW$ arise from the discrepancy
between the small $J$-function and the canonical basis-vector of
solutions of the Picard--Fuchs system: a precise definition and
discussion appears in \cite[App.~A.1.1]{Brini:2013zsa}. Here we only
remark that the functions  $\mathfrak{h}_\cW$ are completely
determined by classical geometric data. Because of the close
relationship between the disk tensor and the Gamma factors of the
central charge in Iritani's theory of integral structures
\cite{MR2553377, Brini:2013zsa},  we have a prediction for
  the transformation $\mathbb{O}$
in terms of the toric geometry of  the targets.

\begin{proposal}[The transformation $\mathbb{O}$]\label{pr:bbo}
Choose a grade restriction window $\mathfrak{W}$ in the GIT problem
to identify the $K$-theory lattices of $\cX$ and $Y$, and for $\cW=
\cX, Y$, define:
\ea{
 \Theta_\cW(\mathbf{1_{p,k}}) &\triangleq  \frac{1}{ \sin\left(\pi\left(\left\langle \sum_{j=1}^l\frac{k_jm_j^3}{n_j} \right\rangle -\frac{w_{3}}{z}\right)\right)}\mathbf{1_{p}^{k}}
\label{Theta}} Then the transformation $\mathbb{O}$ in
Proposal~\ref{theocrc} has the form: \eq{ \bbO=
\Theta_Y\circ\overline{\rm CH}_Y \circ \overline{\rm CH}^{-1}_\X
\circ {\Theta_\X}^{-1}, \label{eq:bbo} } where we denote by
$\overline{\rm CH}_\cW= z^{-\frac{1}{2}\deg }\mathrm{CH}_\cW$ the
matrix of Chern characters  (homogenized with respect to the
cohomological degree ``$\deg$")  in the bases given by
$\mathfrak{W}$.\end{proposal}

In \cite{Brini:2013zsa}, we show that Proposal~\ref{theocrc} follows
from the Coates--Iritani--Tseng's CRC. Proposal~\ref{pr:bbo}
coincides with $\U$ being predicted by a natural equivariant
version\footnote{The fact that $\Gamma$-integral
  structures match with the natural $B$-model integral structures under mirror
  symmetry was proved in \cite{MR2553377} for compact toric orbifolds. A
  general proof of the fully equivariant version of Iritani's $K$-theoretic
  CRC has been
  announced by Coates--Iritani--Jiang.}  of Iritani's $K$-theoretic Crepant
Transformation Conjecture \cite{MR2553377}:

\begin{conj}
\label{conj:iri} For $\cW= \cX, Y$, denote by
$\overline{\Gamma}_\cW$ the diagonal matrix whose $kk$ entry is
$\overline\Gamma_\cW^k$. Then, for every choice $\mathfrak{M}$ of
grade restriction window, there exists a choice of analytic
continuation path $\rho$ such that \eq{\label{eq:iri} \U=
\overline{\Gamma}_Y\circ \overline{\rm CH}_Y \circ \overline{\rm
CH}^{-1}_\X \circ {\overline{\Gamma}^{-1}_\X}. }
\end{conj}

From Proposal~\ref{theocrc} one can extract comparison statements
about generating functions for disk invariants. The strongest
statement can be made when the Lagrangian boundary condition
intersects a leg whose isotropy is preserved in the crepant
transformation.

\begin{proposal}[Scalar disk potentials]\label{sdp}
Let $L$ be a Lagrangian boundary condition on $\cX$ that intersects
a torus invariant line whose generic point has isotropy group $G_L$,
and such that  if we denote $L'$ be the corresponding boundary
condition in $Y$, then $L'$ also intersects a torus invariant line
with generic isotropy group $G_L$. For $\cW= \cX, Y$, define the
{\it  scalar disk potential\footnote{ We choose to define the scalar
disk potential as a generating function for the {\it absolute value}
of disk invariants. In the course of the verifications of Proposal
\ref{sdp}, one may observe that the scalar potentials could be
matched on the nose with the use of appropriate matrices of roots of
unity - that in the end contribute just signs, albeit with some
non-trivial pattern. We have deliberately forgone to keep track of
these phenomena, especially in light of the choice-of-signs the
theory of open invariants is everywhere laden with.}  }: \ea{
F_{\cW}^{\rm disk}(\tau,y,\vec{w}) &=  \sum_d\frac{y^d}{d!}\sum_n
\frac{1}{n!}\left|\langle \tau, \ldots, \tau
\rangle_{0,n}^{\cW,L,d}\right| \triangleq \sum_d
\frac{y^d}{d!}\left|\left( \DD^+_\cW(d;\vec{w}),
J^\cW\left(\tau,\frac{n_{\rm eff}w_1}{d}\right)\right)_{\cW}\right|.
} Then, upon identifying the insertion variables via the change of
variable prescribed by the closed CRC, we have:
\eq{ F_{L',Y}^{\rm disk}(\tau,\mathfrak{h}_Y^{\frac{1}{n_{\rm
eff}w_1}} y,\vec{w}) = F_{L,\X}^{\rm
  disk}(\tau,\mathfrak{h}^{\frac{1}{n_{\rm eff}w_1}} _\cX y,\vec{w}).
}
\end{proposal}


\subsection{Hard Lefschetz targets: the quantized OCRC}

When $\cX$ satisfies the hard Lefschetz condition\footnote{This is
  $\mathrm{age}(\phi)=\mathrm{age}(I^*(\phi))$ for all $\phi \in H_{\rm
    orb}(\cX)$, where $I: I\cX \to I\cX$ is the canonical involution on the
  inertia stack.}, a natural generalization of the CRC to higher genus GW
invariants is achieved by canonical quantization \cite{MR2529944,
  MR3112518}: the all-genus Gromov--Witten
partition functions are viewed as elements of the respective Fock
spaces \cite{MR1866444,MR1901075}, conjecturally matched by the
Weyl-quantization of the classical canonical transformation $\U$.

\begin{conj}[The hard Lefschetz quantized CRC, from \cite{MR2529944, MR3112518}]\label{conj:scr}
Let $\cX \rightarrow X \leftarrow Y$ be a Hard Lefschetz crepant
resolution diagram for which the Coates--Iritani--Tseng CRC holds.
For $\cW$ either $\cX$ or $Y$, let $Z_\cW$ denote the generating
function of disconnected Gromov--Witten invariants of $\cW$ viewed
as an element of the Fock space of $H_{\rm orb}(\cW)\otimes
\bbC((z))$, and $\U$ the Coates--Iritani--Tseng morphism of Givental
spaces identifying the Lagrangian cones of $\cX$ and $Y$. Then
\eq{ Z_Y =  \widehat{\U} Z_\cX \label{eq:sqcrc} } \label{conj:sqcrc}
\end{conj}

In the context of torus-equivariant Gromov--Witten theory of
orbifolds with zero-dimensional fixed loci, the hard Lefschetz
quantized CRC can be proven in two steps
\cite[Prop.~6.3]{Brini:2013zsa}, as follows.

\begin{enumerate}
\item[\rm (1)] Combining the Coates--Givental/Tseng quantum Riemann--Roch theorem
  \cite{MR2276766,MR2578300} with Givental's quantization formula in a neighborhood of the
  large radius points of $\cW$ identifies a ``canonical"
  $R$-calibration defined locally by the genus $0$ GW theory of $\cW$;
\item[\rm (2)]  Conjecture \ref{conj:sqcrc} then follows from establishing the equality, upon
  analytic continuation, of the canonical $R$-calibrations of $\cX$ and $Y$ on the
  locus where the quantum product is semi-simple.
\end{enumerate}

The main consequence drawn in \cite{Brini:2013zsa} for open
Gromov--Witten invariants is a CRC statement for all genera and
number of holes.

\begin{proposal}[The quantized OCRC \cite{Brini:2013zsa}]
\label{conj:manyholes} Let $\X \rightarrow X \leftarrow Y$ be a Hard
Lefschetz diagram for which the higher genus closed CRC holds.
 Define the \textit{genus g, $\ell$-holes winding neutral potential}
 $\mathbb{F}^{g,\ell}_{\cW, L}:H(\cW)\to \HH_\cW^{\otimes \ell}$ as
%
\eq{ \mathbb{F}^{g,\ell}_{\cW, L}(\tau,z_1,\ldots,z_\ell,\vec{w})
\triangleq \overline{\DD}_\cW^{+\otimes \ell} \circ
J^\cW_{g,\ell}\left(\tau,z_1,\ldots, z_\ell;\vec{w}\right),
\label{mholepot} }
where $J^\cW_{g,\ell}$ encodes  genus $g$, $\ell$-point descendent
invariants: \eq{ J^\cW_{g,\ell}(\tau,\bfz; \vec{w})\triangleq
\bra\bra \frac{\phi_{\alpha_1}}{z_1-\psi_1}, \dots,
\frac{\phi_{\alpha_\ell}}{z_\ell-\psi_\ell}
\ket\ket_{g,\ell}\phi^{\alpha_1}\otimes\cdots\otimes\phi^{\alpha_\ell}.
\label{eq:JZn} }
Further, let  $\bbO^{\otimes \ell}= \bbO(z_1)\otimes\ldots\otimes
\bbO(z_\ell)$. Then, \eq{ {\mathbb{F}_{L',Y}^{g,\ell}}=
\bbO^{\otimes \ell}\circ \mathbb{F}_{L,\X}^{g,\ell}.
 }
\end{proposal}

\section{Example 1: local weighted projective planes}
\label{sec:localpp}

\subsection{Classical geometry}

The family of geometries we study arises as the GIT quotient \eq{
\bbC^4/\!\!/_\chi \ \bbC^\star, } with torus action  on the
coordinates $(x_1, x_2,x_3,x_4)$ specified by the charge matrix \eq{
M = \l(\bary{cccc} n & 1 & 1 & -2-n \eary \r). }
The quotients obtained as the character $\chi$ varies are the toric
varieties whose fans are represented in Figure \ref{fig:fanlp11neven}.
%
\begin{figure}[b]
\centerline{\includegraphics{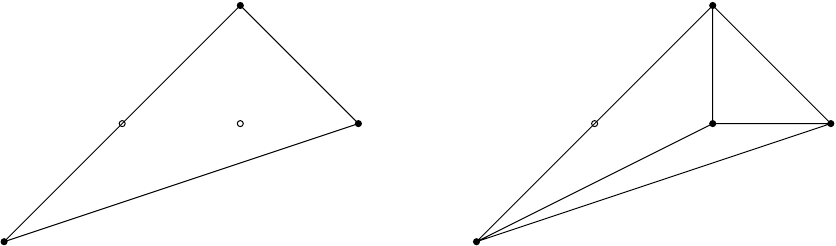}}
\caption{ A height $1$ slice of the fans of  $[\bbC^3/\bbZ_{n+2}]$
(left) and local
  $\bbP(n,1,1)$ (right) for $n=2$.}
\label{fig:fanlp11neven}
\end{figure}
The right hand side of Figure \ref{fig:fanlp11neven} corresponds to $\chi
> 0$ . The irrelevant ideal is

%
%
%
\eq{ \cI_{\rm LR} \triangleq \bra x_1, x_2, x_3\ket }
and the resulting geometry $Y$ is the total space of
$\cO(-n-2)_{\bbP(n,1,1)}$; $[x_1:x_2:x_3]$ serve as
(quasi)-homogeneous coordinates for the base, and $x_4$ is an affine
fiber coordinate. Torus fixed points and invariant lines are: \ea{
\label{eq:L1pn11} L_1 =& V(x_1, x_4),&
L_2  = & V(x_2, x_4), &
L_3 = & V(x_3, x_4), \\
\label{eq:P1pn11} P_1 =& V(x_2, x_3, x_4), &
P_2  = & V(x_1, x_3, x_4), &
P_3 = & V(x_1, x_2, x_4). }
We have $L_1\simeq \bbP^1$, $L_2,L_3\simeq \bbP(1,n)$, $P_2$, $P_3
\simeq
[\mathrm{pt}]$, $P_1\simeq B\bbZ_n$. The fibers over the fixed points $P_2$ and $P_3$ are non-gerby. The fiber over $P_1$ is non-gerby when when $n$ is odd; when $n$ is even, it has a  $\bbZ_2$-subgroup as a stabilizer. \\

When $\chi $ is negative we have the fan on left hand side of
Figure \ref{fig:fanlp11neven}, which gives the irrelevant ideal
%
%
%
\eq{ \cI_{\rm OP} \triangleq \bra x_4\ket. }
Quotienting by $x_4 \neq 0$ gives a residual $\bbZ_{n+2}$ action on
$\bbC^3$ with weights $(n,1,1)$; the resulting orbifold
$[\bbC^3/\bbZ_{n+2}]$ will be denoted by $\cX$. Moving across
$\chi=0$
\eq{ \l[\bary{c}
x_1 \\
x_2 \\
x_3 \\
x_4 \\
\eary \r] \in \bbC^4/\!\!/\bbC^* \rightarrow \l[\bary{c}
x_1 x_4^{\frac{n}{n+2}} \\
x_2 x_4^{\frac{1}{n+2}}\\
x_3 x_4^{\frac{1}{n+2}} \\
\eary \r]\in \bbC^3/\bbZ_{n+2} }
where we denoted by $[x_1, \dots, x_n]$ the equivalence class in the
appropriate quotient, is a birational contraction of the image of
the zero section $s:\bbP(n,1,1)\hookrightarrow K_{\bbP(n,1,1)}$.

\subsubsection{Bases for cohomology}
We consider a Calabi--Yau $2$-torus action on $Y$ and $\cX$,
descending from an action on $\bbC^4$  with geometric weights
$(\alpha_1, \alpha_2, -(\a_1+\a_2),0)$. Note that we consider the
geometric weights as elements  of $H^2(BT)$: an integer $\alpha$
corresponds to the first Chern class of the representation $t
\mapsto t^\alpha$.  The tangent weights at the torus fixed points
are depicted
in the toric diagrams in Figure \ref{fig:td}. \\

\begin{figure}[t]
\centerline{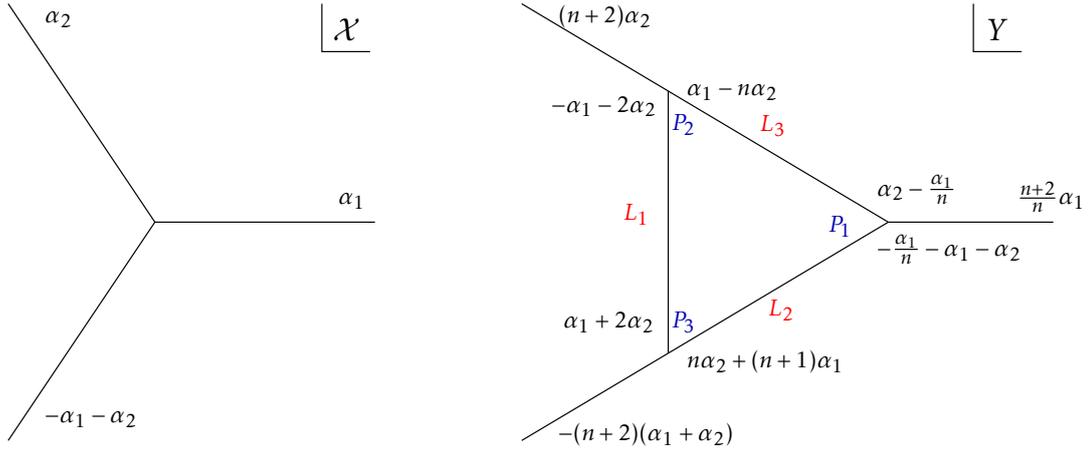}
\caption{Toric web diagrams and weights at the fixed points for
$\cX$ and $Y$.} \label{fig:td}
\end{figure}

Let $p=\pi^*c_1(\cO_{\bbP(n,1,1)}(1))\in H_T(K_{\bbP(n,1,1)})$,
where $\pi:K_{\bbP(n,1,1)} \to \bbP(n,1,1)$ is the bundle projection
and the torus action on $\cO_{\bbP(n,1,1)}(1)$ is linearized
canonically by identifying $\bbC^4$ with the tautological bundle
$\cO_{\bbP(n,1,1)}(-1)$. Via the Atiyah--Bott isomorphism we have:
\eq{ p = -\frac{\a_1}{n} P_1-\a_2 P_2 +(\a_1+\a_2) P_3 \in
H^2_T(K_{\bbP(n,1,1)}). }
The products $w_i$ of the three normal (tangent) weights at  the
fixed points $P_i$ read
\ea{
w_1 = & -\frac{n+2}{n}\a_1\l(\a_2-\frac{\a_1}{n}\r)\l(\a_1+\a_2+\frac{\a_1}{n}\r),\nn\\
w_2 = & -(n+2)\a_2(\a_1-n \a_2)(\a_1+2\a_2), \nn \\
w_3 = & -(n+2)(\a_1+\a_2)(\a_1+n (\a_1+\a_2))(\a_1 + 2\a_2). }
As a module over $H(BT)$, the equivariant Chen--Ruan cohomology ring
of $Y=K_{\bbP(n,1,1)}$ is spanned by $\{\mathbf{1}_Y, p, p^2,
\mathbf{1}_{\frac{1}{n}}, \dots, \mathbf{1}_{\frac{n-1}{n}}\}$. On
$\cX$, we have cohomology classes $\mathbf{1}_g$, labeled by the
corresponding group elements $g=1,\re^{2\pi\ri/{n+2}}, \dots,
\re^{2\pi\ri (n+1)/(n+2)}$; the involution on the inertia stack
exchanges $\mathbf{1}_{\frac{k}{n+2}}\leftrightarrow
\mathbf{1}_{1-\frac{k}{n+2}}$.

\subsection{Quantum geometry}

Genus-zero Gromov--Witten invariants of $\cX$ and $Y$ can be
computed using the quantum Riemann--Roch theorems of
Coates--Givental \cite{MR2276766} and Tseng \cite{MR2578300} applied
to the Gromov--Witten theories of $B\bbZ_{n+2}$ and $\bbP(n,1,1)$,
respectively. We have the following

\begin{prop}[\cite{MR2276766,MR2578300,MR2510741}]
For $|y|<n^n(n+2)^{-2-n}$,$|x|<(n+2)n^{-n/(n+2)}$, define the {\rm
$I$-functions} \ea{ \label{eq:IY} I^Y(y,z)\triangleq & z y^{p/z}
\sum_{nd \in \bbZ^+} y^d \frac{\prod_{\stackrel{\bra m \ket =
      \bra (n+2) d\ket }{0\leq m <(n+2) d-1}} (-(n+2)p-m z)}{\prod_{\stackrel{\bra m \ket =
      \bra d \ket }{0<m \leq d}} (p+\a_2+mz)(p-\a_1-\a_2+m
  z)}\frac{1}{\prod_{m=1}^{nd}(np+\a_1+m z)}, \\
I^\cX(x,z) \triangleq  & \sum_{k\geq 0} \frac{\prod_{\stackrel{\bra
b \ket = \bra k/(n+2) \ket}{0\leq
      b<\frac{k}{n+2}}}(\frac{\alpha_2}{n+2}-b z)(-\frac{\alpha_1+\alpha_2}{n+2}-b z)\prod_{\stackrel{\bra b \ket = \bra kn/(n+2)\ket}{0\leq
      b<\frac{kn}{n+2}}}(\frac{n \alpha_1}{n+2}-b z)}{z^k}\frac{x^k}{k!}\mathbf{1}_{\bra k/{n+2}\ket}.
\label{eq:IX} } Then, for $\cW$ either $\cX$ or $Y$ and $w$ either
$x$ or $y$, $I^\cW(w,-z)\in -z+H_T(\cW) \otimes \bbC[[z^{-1}]]\cap
\LL_\cW$ identically in $w$. \label{prop:Ifun}
\end{prop}
\begin{proof}
This is \cite[Theorem~3.5 and 3.7]{MR2486673}.
\hfill $\Box$ \end{proof}

Since the $I$-functions of $\cX$ and $Y$ belong to the cone and
behave like $z+\cO(1)$ at large $z$, they coincide with suitable
restrictions of the respective big $J$-functions to a subfamily of
quantum cohomology parameters.
\begin{cor}
Denote by $q$ the Novikov variable associated to $p$ and write
$\phi=\sum_{k=0}^{n+1}
\tau_{\frac{k}{n+2}}\mathbf{1}_{\frac{k}{n+2}}$ for an orbifold
cohomology class $\phi\in H_T^{\rm orb}(\cX)$. Then the following
equalities hold: \ea{
J_{\rm small}^Y(q,z) = & I^Y(y(q),z), \\
J_{\rm
  big}^\cX(\phi,z)\big|_{\tau_{k/(n+2)}=  \delta_{k1}\tau}
= & I^Y(x(\tau),z), }
where $\log q = \lim_{z\to\infty}(I^Y(y,z)-z), \tau =
\lim_{z\to\infty}(I^\cX(x,z)-z)$. In particular,
\eq{ \mathfrak{h}_Y = \mathfrak{h}_\cX = 1. }
\end{cor}
\subsubsection{Analytic continuation and $\U$}
A standard method \cite{MR2529944,MR2510741} to relate the
Lagrangian cones of $\cX$ and $Y$ upon analytic continuation hinges
on the following three-step procedure: \ben
\item[\rm (1)] find a holonomic linear differential system of rank equal to $\dim
  H^\bullet(Y)=\dim H^\bullet_{\rm orb}(\cX)$ jointly satisfied, upon appropriate identification of the
quantum parameters, by the components of the $I$-functions of $\cX$
and $Y$ as convergent power series around the respective boundary
point;
\item[\rm (2)] determine the relation between the $I$-functions upon analytic continuation along a path
  $\rho$ connecting the two boundary points;
\item[\rm (3)] invoke a reconstruction theorem to recover from the latter the content of big quantum
cohomology and the full-descendent theory in genus zero
\cite{ccit2,Dubrovin:1994hc}. \een Step (3) has been achieved in
full generality for toric Deligne--Mumford stacks in \cite{ccit2}.
The first step is also standard \cite{MR1672116}; we spell out the
details below for the sake of completeness. The main intricacy here
lies in Step~(2), as the rank of the system is parametrically large
in $n$ and the usual Mellin--Barnes method
\cite{MR2486673,MR2700280} is technically more subtle to apply; we
present a workaround in the discussion leading to Lemma \ref{lem:kum}.

\begin{lem}
\label{lem:pf} Let $\DD_Y$ the $(n+2)^{\rm th}$ order linear
differential operator
\eq{ \label{eq:DY} \DD_Y\triangleq
(\theta_y+\a_2)(\theta_y-\a_1-\a_2) \prod_{m=0}^{n}(n\theta_y+\a_1
-mz)- y \prod_{m=0}^{n+1}(-(n+2) \theta_y-m z) } where
$\theta_y=zy\de_{y}$ and define $\DD_\cX$ to be the differential
operator obtained by replacing $y=x^{-n-2}$ in Eq. \cref{eq:DY}.
Then, \eq{ \DD_\bullet I^\bullet = 0 }
\end{lem}
\begin{proof}
The statement follows from a straightforward calculation from Eqs.
\cref{eq:IY} and \cref{eq:IX}.
\hfill $\Box$ \end{proof}

The linear operator $\DD_\cW$ is the Picard--Fuchs operator of
$\cW=\cX,Y$: Lemma \ref{lem:pf} establishes that the torus-localized
components of the $I$-functions of $\cX$ and $Y$ furnish two bases
solutions of the linear system $\DD_\cW f=0$, respectively in the
neighbourhood of the Fuchsian points $y=0$ and $\infty$.  Relating
the cones of $\cX$ and $Y$ thus boils down to
finding the change-of-basis matrix connecting the two set of
solutions upon
analytic continuation from one boundary point to the other. 
Let $I^\cX_{k}(x,z)$ denote the coefficient of
$\mathbf{1}_{k/(n+2)}$ in Eq. \cref{eq:IX}, and define in the same
vein
\ea{
I_k^Y(y,z)= & \mathrm{Coeff}_{\mathbf{1}_{P_{k+1}}}I^Y(y,z), & & k=0,1,2, \\
I_{\frac{j}{n}}^Y(y,z)= &
\mathrm{Coeff}_{\mathbf{1}_{\frac{j}{n}}}I^Y(y,z), & &
j=1,\dots,n-1. }
It is immediately noticed that
$I_k^\cX(x,z)=x^k(z^{1-k}/k!+\cO(x^{n+2}))$: this uniquely
characterizes $\{I_k^\cX\}_{k=0}^{n+1}$ as a basis of solutions of
$\DD_\cX f=0$. On the other hand, localizing Eq. \cref{eq:IY} to the
$T$-fixed points and resumming in $d$ for
$|y|<\frac{n^n}{(n+2)^{n+2}}$ we obtain
\ea{ \label{eq:Ihypun} I^Y_k = & i^*_{P_k}\l[ z y^{p/z} \,
_{n+3}F_{n+2}\left(\{A_n\}; \{B_n\};
  (-n-2)^{n+2}n^{-n} y\right)\r], \\
I^Y_{\frac{j}{n}} = & \frac{z^{1-j} y^{j/n}}{j!} \,
_{n+2}F_{n+1}\left(\{C_{n,j}\}; \{D_{n,j}\}; (-n-2)^{n+2}n^{-n}
y\right), \label{eq:Ihyptw} }
where
\ea{ A_n= & \l(1, \frac{1}{n+2}+\frac{p}{z},\dots,
\frac{n+1}{n+2}+\frac{p}{z},
\frac{p}{z}\r), \nn \\
B_n = & \l(\frac{1}{n}+\frac{np+\a_1}{n z}, \dots,
\frac{n-1}{n}+\frac{np+\a_1}{n z}, 1+\frac{np+\a_1}{n z},
1+\frac{p-\alpha
  _1-\a_2}{z},1+\frac{p+\a_2}{z}\r), \nn \\
C_{n,j}= & \l(1,\frac{1}{n+2}-\frac{j}{n},\dots,
\frac{n+1}{n+2}-\frac{j}{n}\r), D_{n,j} =  \l(\frac{j}{n},
\frac{j+1}{n}, \dots, \frac{j+n-1}{n}, 1+\frac{j}{n}\r),
\label{eq:ABCD} }
and $_{p}F_{q}\left(\{\cA\}; \{\cB\}; y\right)$ denotes the
generalized hypergeometric series
\eq{ _{p}F_{q}\left(\{\cA\}; \{\cB\}; w\right) \triangleq
\frac{\prod_{i=1}^q \Gamma(\cB_i)}{\prod_{j=1}^p \Gamma(\cA_j)}
  \sum_{n=0}^{\infty} \frac{\prod_{i=1}^p \Gamma(\cA_i+n)}{\prod_{j=1}^q \Gamma(\cB_j+n)}\frac{w^n}{n!},
\label{eq:pFq} }
which is convergent for $|w|<1$. \\

In order to continue to $x=y^{-n-2}\ll 1$ we will need the following analytic
continuation theorem for  $_{p}F_{q}\left(\{\cA\}; \{\cB\};
y\right)$, which generalizes the
 classical Kummer continuation formula at infinity for the Gauss function.
\begin{lem}
\label{lem:kum} Let $p=q+1$, $\cB_j\notin \bbN$, $\cA_i-\cA_j \notin
\bbZ$ for $i\neq j$  and let $\rho:\bbR \to \bbC$ be a path in the
complex $y$-plane from $y=0$ to $y=\infty$ having trivial winding
number around both $y=0$ and $y=1$. Then the analytic continuation
of Eq. \cref{eq:pFq} to $|y| \gg 1$ along $\rho$ satisfies \eq{ \,
_{q+1}F_{q}\left(\{\cA\}; \{\cB\}; y\right) \sim
\sum_{k=1}^{q+1}\prod_{j=1}^q
\frac{\Gamma(\cB_j)}{\Gamma(\cB_j-\cA_k)}\prod_{j\neq
k}\frac{\Gamma(\cA_j-\cA_k)}{\Gamma(\cA_j)}(-y)^{-\cA_k}\l(1+\cO\l(\frac{1}{y}\r)\r).
\label{eq:kum} }
\end{lem}
\begin{proof}
The argument follows almost verbatim the steps leading to the
well-known result for $q=1$. $\Phi(w) \triangleq
_{q+1}F_{q}\left(\{\cA\}; \{\cB\}; w\right)$ satisfies the
generalized hypergeometric equation
\eq{ \l[\theta \prod_{j=1}^q(\theta+\cB_j-1) - w
\prod_{j=1}^q(\theta+\cA_j)\r] \Phi (w)=0. \label{eq:hypeq} }
with $\theta=w\de_w$. The same analysis at $w=\infty$ as for the
Gauss equation reveals that $\cA_i$ are local exponents of Eq.
\cref{eq:hypeq},
%
%
%
\eq{ \widetilde \Phi(w) \sim \sum_{j=1}^{q+1} c_j\left(\{\cA\};
\{\cB\}\right) (-w)^{-\cA_j} \label{eq:phipsi} }
for some $c_j\left(\{\cA\}; \{\cB\}\right) \in \bbC$. Let now $k$ be
such that $\Re(\cA_k-\cA_j)<0$ for all $j\neq k$; then
\eq{
 c_k\left(\{\cA\}; \{\cB\}\right)= \lim_{w\to\infty} (-w)^{\cA_k} \widetilde \Phi(w)
}
Now, $\Phi(w)$ can be represented as the multiple Euler--Pochhammer
integral \cite{MR0474684}
\eq{ \Phi_j(w)= \prod_{i=1}^q
\frac{\Gamma(\cB_i)}{\Gamma(\cA_i)\Gamma(\cB_i-\cA_i)}\frac{1}{(1-\re^{2\pi\ri
  \cA_i})(1-\re^{2\pi\ri
  (\cB_i-\cA_i)})}\int_\gamma\dots
\int_\gamma\frac{t_i^{\cA_i}(1-t_i)^{\cB_i-\cA_i}}{(1-w\prod_i
t_i)}\prod_{i=1}^q \frac{\rd t_i}{t_i(1-t_i)}, }
where $\gamma=[C_0,C_1]$ is the commutator of simple loops around
$t=0$ and $t=1$. Taking the limit $w\to \infty$ along $\rho$ and
using the Euler Beta integral,
\eq{ \frac{1}{(1-\re^{2\pi\ri  \cA_i})(1-\re^{2\pi\ri
(\cB_i-\cA_i)})}
\int_\gamma t_i^{\cA_i-1}(1-t_i)^{\cB_i-\cA_i-1}\prod_{i=1}^q \rd t_i 
= \frac{\Gamma(\cA_i)\Gamma(\cB_i-\cA_i)}{\Gamma(\cB_i)},
\label{eq:EulerBeta} }
gives
\eq{ c_k(\cA,\cB)=\prod_{i=1}^q
\frac{\Gamma(\cB_i)}{\Gamma(\cB_i-\cA_k)}\prod_{i\neq
k}\frac{\Gamma(\cA_i-\cA_k)}{\Gamma(\cA_i)}. } from which Eq.
\cref{eq:kum} follows by the invariance of Eq. \cref{eq:pFq} under
permutation of $\cA_i$ and analytic continuation to
$\Re(\cA_j-\cA_i)<0$, $j\neq k \neq i$.
\hfill $\Box$ \end{proof}

Denote by $\widetilde{I^Y}(y,z)$ the analytic continuation of
$I^Y(y,z)$ along  $\rho$ as in Lemma \ref{lem:kum}. The matrix
expression of the symplectomorphism $\U:\HH_\cX \to \HH_Y$ of
Conjecture \ref{conj:iri} in the bases
$\{\mathbf{1}_{\frac{k}{n+2}}\}_{k=0,1,\ldots, n-1}$ for
$H_T^\bullet(\cX)$ and $\{P_1,\mathbf{1}_{\frac{1}{n}} , \ldots,
\mathbf{1}_{\frac{n-1}{n}},P_2, P_3\}$ for $H_T^\bullet(Y)$ can then
be read off upon applying  Eq. \cref{eq:kum} to Eqs.
\cref{eq:Ihypun}--\cref{eq:ABCD},
\eq{ \widetilde{I}_i^Y(x^{-n-2},z) = \sum_{k=0}^{n+1}(\U)_{ik}
I_{\frac{k}{n+2}}^\cX(x,z). }
%
\begin{example}[$n=1$]{\rm
We have, from Eq. \cref{eq:IY} and Eq. \cref{eq:kum} for $q=2$,
\ea{ (\U)_{0,0}= & \frac{\Gamma \left(\frac{1}{3}\right) \Gamma
\left(\frac{2}{3}\right) 27^{\frac{\alpha _2}{z}} \Gamma
\left(\frac{z+\alpha _1-\alpha _2}{z}\right) \Gamma
   \left(\frac{z-\alpha _2+\alpha _3}{z}\right)}{\Gamma \left(\frac{z+\alpha _1}{z}\right) \Gamma \left(\frac{1}{3}-\frac{\alpha _2}{z}\right) \Gamma
   \left(\frac{2}{3}-\frac{\alpha _2}{z}\right) \Gamma \left(\frac{z+\alpha
     _3}{z}\right)}, \nn \\
(\U)_{0,\frac{1}{3}} = & \frac{z \Gamma \left(-\frac{1}{3}\right)
\Gamma \left(\frac{1}{3}\right) 3^{\frac{3 \alpha _2}{z}-1} \Gamma
\left(\frac{z+\alpha _1-\alpha _2}{z}\right) \Gamma
   \left(\frac{z-\alpha _2+\alpha _3}{z}\right)}{\Gamma \left(\frac{\alpha _1}{z}+\frac{2}{3}\right) \Gamma \left(-\frac{\alpha _2}{z}\right) \Gamma
   \left(\frac{2}{3}-\frac{\alpha _2}{z}\right) \Gamma \left(\frac{\alpha
     _3}{z}+\frac{2}{3}\right)}, \nn \\
(\U)_{0,\frac{2}{3}}= & \frac{2 z^2 \Gamma \left(-\frac{2}{3}\right)
\Gamma \left(-\frac{1}{3}\right) 3^{\frac{3 \alpha _2}{z}-2} \Gamma
\left(\frac{z+\alpha _1-\alpha _2}{z}\right)
   \Gamma \left(\frac{z-\alpha _2+\alpha _3}{z}\right)}{\Gamma \left(\frac{\alpha _1}{z}+\frac{1}{3}\right) \Gamma \left(-\frac{\alpha _2}{z}\right) \Gamma
   \left(\frac{1}{3}-\frac{\alpha _2}{z}\right) \Gamma \left(\frac{\alpha _3}{z}+\frac{1}{3}\right)},
}
where $\a_3=-\a_1-\a_2$, and
$(\U)_{ik}(\a_{(1,2,3)})=(\U)_{0k}(\a_{\chi^{i}(1,2,3)})$, where
$\chi\in S_3$ is the cyclic permutation $1\to 2$, $2\to 3$, $3 \to
1$.}
\end{example}

\begin{rmk}[On general toric wall-crossings]{\rm
The arguments we used for the examples of this Section have a wider
applicability to wall-crossings in toric Gromov--Witten theory,
including the multi-parameter case. On general grounds,
$I$-functions - and their extended versions \cite{ccit2} - are
multiple hypergeometric functions of Horn type
\cite{MR2700280,MR2553377}. When crossing a {\it single} wall in the
$B$-model moduli space, however, the analytic continuation is
effectively taking place in {\it one} parameter only. Restricting to
the sublocus where all the spectator variables are set to zero
reduces the multiple Horn series to a single-variable series which,
upon manipulations of Gamma factors in the summand as in the next
section, can always be cast in the form of a generalized
hypergeometric function $_p F_q(\{\cA\},\{\cB\},w)$ with $q\geq
p-1$. Whenever the series has a finite radius of convergence
 as in the Calabi--Yau case, we have  $p=q+1$, for which
Lemma \ref{lem:kum} applies. The general case is obtained similarly.}
\end{rmk}

\subsubsection{Grade restriction window and the $K$-theoretic CRC}
Let us now turn to Conjecture \ref{conj:iri} for this family of
geometries. Throughout this section, we work with the natural basis
$\{\mathbf{1}_{\frac{k}{n+2}}\}_{k=0,1,\ldots, n-1}$ for
$H_T^\bullet(\cX)$ and with the localized basis
$\{P_1,\mathbf{1}_{\frac{1}{n}} , \ldots,
\mathbf{1}_{\frac{n-1}{n}},P_2, P_3\}$ for   $H_T^\bullet(Y)$. The
grade restriction window  $\mathfrak{W}=\{L_j\}_{j=0,\ldots,n+1}$,
where $L_j$ is a $\bbC^\ast$ equivariant line bundle on $\bbC^4$
with character $\chi_j$ given by
\eq{ \chi_j=
\begin{cases}
j & j<1+\frac{n}{2}, \\
j-n-2 & \mathrm{else},
\end{cases}
}
yields a natural bijection between the $K$-lattices of $\cX$ and
$Y$.  We make the notational convention of taking all indexing  sets
to range from $0$ to $n+1$, with the sole purpose of leaving the
coefficients corresponding to identities/trivial objects in the
first row/column of any matrix we write. With these choices the
matrices representing the (homogenized, involution pulled-back)
Chern characters for $\cX$ and $Y$ are \ea{ \label{eq:chX}
[{\overline{\rm CH}_{\cX}}]_j^k = & {\left(\frac{2\pi
    \rm{i}}{z}\right)}^{\frac{1}{2}\deg }\mathrm{inv}^\ast \mathrm{CH}_\cX=
{\rm{e}}^{-jk \frac{2\pi \rm{i}}{n+2}}, \\
[{\overline{\rm CH}_{Y}}]_j^l= & \left\{
\begin{array}{cl}
{\rm{e}}^{\frac{2\pi \rm{i}}{n} \chi_j\left(l- \frac{\a_1}{z}\right)} & \mbox{for $l= 0, \ldots ,n-1$.}\\
{\rm{e}}^{-2\pi \rm{i} \frac{\chi_j\a_2}{z}} & \mbox{for $l= n$.}\\
{\rm{e}}^{-2\pi \rm{i} \frac{\chi_j\a_3}{z}}  & \mbox{for $l= n+1$.}
\end{array}
\right. \label{eq:chY} }

\begin{thm}
Conjecture \ref{conj:iri} holds with the restriction window
$\mathfrak{W}$ above and the analytic continuation path $\rho$ as in
Lemma \ref{lem:kum}. \label{thm:crcproj}
\end{thm}
\begin{proof}
Consider the linear map $\bbV:\HH_\cX \to \HH_Y$ defined by \eq{
\bbV = \Gamma_Y^{-1} \U \Gamma_\cX, \label{eq:V} } in the bases
above for $H_T^\bullet(\cX)$ and $H_T^\bullet(Y)$. The Gamma factors
in Eqs. \cref{eq:kum} and \cref{eq:V} telescope away by virtue of
Eq. \cref{eq:ABCD}, the multiplication formula
\eq{ \Gamma (b+m z)=(2 \pi )^{\frac{1-m}{2}} m^{b+m z-\frac{1}{2}}
\prod _{k=0}^{m-1} \Gamma \left(\frac{b+k}{m}+z\right)\text{;} \ \
m\in \mathbb{Z}\land m>0, }
and Euler's identity, $\Gamma(x)\Gamma(1-x)=\pi/\sin(\pi x)$; the
final result is a trigonometric matrix with coefficients
$[\bbV]_j^i$ being
 Laurent polynomials in
$\re^{2\pi\ri \a_k}$, $k=1,2,3$. Right-multiplication by the Chern
character matrix of $\cX$ and telescoping the resulting sums over
roots of unity returns ${\overline{\rm CH}_{Y}}$, as given in Eq.
\cref{eq:chY}.
\hfill $\Box$ \end{proof}

\subsection{The OCRC}


As discussed in Section \ref{sec:dcrc}, the first implication we
draw from Theorem \ref{thm:crcproj} is a comparison theorem for
winding neutral disk potentials.

\begin{cor}
Proposals \ref{theocrc}  and \ref{pr:bbo} hold for
$Y=K_{\bbP(n,1,1)}$ and $\cX=[\bbC^3/\bbZ_{n+2}]$.
\label{cor:crcproj}
\end{cor}

This can be employed to obtain more concrete identifications of
scalar disk potentials, as we now show.

\subsubsection{Scalar disk potentials: non-special legs}

  In the case where the Lagrangian on $Y$ is on a leg that attached to a
 non-stacky point, the equality of scalar disk potentials follows in a simple
 fashion for all $n$.
 When the Lagrangian is on the leg that attached to the
 stacky point, we need to consider separately the case $n$-odd, where the
 quotient on the leg is effective, and $n$-even, where there is a residual
 $\bbZ_2$ isotropy. \\

We consider non-special legs first. We have the following

\begin{thm}
Consider a Lagrangian boundary condition $L$ on $\cX$ which
intersects the second coordinate axis, and denote by $L'$ the proper
transform in $Y$. Then, upon identifying the insertion variables via
the change of variable prescribed by the closed CRC, we have the
equality of scalar disk potentials:
\eq{ F_{L',Y}^{\rm disk}(\tau,
y,\vec{w}) =  F_{L,\X}^{\rm
  disk}(\tau,
y, \vec{w}). } \label{thm:nonspecleg}
\end{thm}

\begin{proof}
In this case the tensors $\Theta$ from (\ref{Theta})  are: \eq{
\left[\Theta_\cX^{-1}\right]^{kk} = \sin\left(\pi\left(
\frac{-\a_1}{z} + \left\langle\frac{nk}{n+2} \right\rangle\right)
\right), } \eq{ \left[\Theta_Y\right]_{ll} =
\frac{1}{\sin\left(\pi\left( \frac{n\a_2-\a_1}{z}\right) \right)}
\delta_{l,n}. } We compute the transformation $\bbO$ as in Eq.
\cref{eq:bbo}; note it has nonzero coefficients only for $l=n$. We
then specialize $z=\frac{(n+2)\a_2}{d}$ to obtain a map we denote
$\bbO_d$, \eq{ \bbO^k_{d,n} = \frac{\sin\left(\pi\left( -\frac{\a_1
d}{(n+2) \a_2} + \left\langle\frac{nk}{n+2} \right\rangle\right)
\right)}{\sin\left(\pi\left(
 -\frac{\a_1 d}{(n+2) \a_2} + \frac{n d}{(n+2)}
\right) \right)}
 \frac{1}{n+2}
 {\rm e}^{\frac{2\pi {\rm i} j}{n+2}\left(
 k -d
 \right)}.
 \label{eq:od1}}
The expression in Eq. \cref{eq:od1} is summed over the index $j$
ranging from $0$ to $n+1$. When $k$ is not congruent to $d$ modulo
$n+2$, the exponential part is a sum of roots of unity that adds to
$0$. When $k\equiv d$ modulo $n+2$, $\bbO^k_{d,n}=\pm1$. Hence our
OCRC, Corollary \ref{cor:crcproj}, together with Eq. \cref{eq:od1}
gives \eq{
\pm \bbF_{L,\X|{z=\frac{(n+2)\a_2}{d}}}^{\rm
 disk}(\mathbf{1}_{\left\langle\frac{d}{n+2}\right\rangle})=
\bbF_{L',Y|{z=\frac{(n+2)\a_2}{d}}}^{\rm
 disk}(P_2).
}
Disk invariants of winding $d$  for $\cX$ are  the coefficients of
the classes $\mathbf{1}^{\frac{k}{n+2}}$ with   $k\equiv d$ modulo
$n+2$ after specializing $z=\frac{(n+2)\a_2}{d}$ in
$\bbF_{L,\X}^{\rm
 disk}$. 
Summing over all $d$, we obtain the equality of scalar potentials as
stated in Theorem \ref{thm:nonspecleg}.
\hfill $\Box$ \end{proof}

\subsubsection{Scalar disk potentials for the special leg: $n$ odd}
\label{sec:nodd}
\begin{thm}
Let $n$ be an odd integer. Consider a Lagrangian boundary condition
$L$ on $\cX$ which intersects the first coordinate axis, and denote
by $L'$ the proper transform in $Y$. Then, upon identifying the
insertion variables via the change of variable prescribed by the
closed CRC, we have the equality of scalar disk potentials:
\eq{ F_{L',Y}^{\rm disk}(\tau,
y,\vec{w}) =  F_{L,\X}^{\rm
  disk}(\tau,
y, \vec{w}). } \label{thm:speclegodd}
\end{thm}

\begin{proof}
In this case the tensors $\Theta$ from (\ref{Theta}) are: \eq{
\left[\Theta_\cX^{-1}\right]^{kk} = \sin\left(\pi\left(
\frac{\a_1+\a_2}{z} + \left\langle\frac{k}{n+2} \right\rangle\right)
\right), } \eq{ \left[\Theta_Y\right]_{ll} =
\frac{1}{\sin\left(\pi\left( \frac{\a_1+\a_2}{z}+ \frac{\a_1}{nz}+
\left\langle -\frac{l}{n} \right\rangle\right)
 \right)}.
} We compute the transformation $\bbO$ as in Eq. \cref{eq:bbo}. We
then specialize $z=\frac{(n+2)\a_1}{d}$ to obtain $\bbO_d$. \eq{
\bbO^k_{d,l} =
 \frac{
 \sin\left(\pi\left( \frac{d(\a_1+\a_2)}{(n+2)\a_1} + \left\langle\frac{k}{n+2} \right\rangle\right) \right)
} {\sin\left(\pi\left( \frac{d(\a_1+\a_2)}{(n+2)\a_1}+
\frac{d}{n(n+2)}+ \left\langle -\frac{l}{n} \right\rangle\right)
 \right)}
 \frac{1}{n+2}
 {\rm e}^{\frac{2\pi {\rm i} j}{n(n+2)}\left(
 kn +l(n+2) -d
 \right)}.
 \label{eq:od2}}
The expression in Eq. \cref{eq:od2} is summed over the index $j$
ranging from $0$ to $n+1$.  The degree-twisting compatibilities are:
\begin{itemize}
    \item[$\X$:] $ d\equiv kn \ \ \ \mbox{mod}\ \  n+2 ,$
    \item[$Y$:] $ d\equiv 2l \ \ \ \mbox{mod} \ \ n.$
\end{itemize}
The Chinese remainder theorem then states that both compatibilities
are satisfied when
\eq{ d\equiv kn +l(n+2) \ \ \ \mbox{mod} \ \  n(n+2). \label{crt} }
When (\ref{crt}) is satisfied,  the difference in the arguments in
the sine functions is an integer multiple of $\pi$, hence
$\bbO^k_{d,l}=\pm1$. When only the compatibility for $Y$ is
satisfied, then the exponential part of Eq. \cref{eq:od2} consists
of a sum of $(n+2)$ roots of unity that add to $0$. All other
entries of the matrix representing $\bbO_d$ do not matter for our
purposes. For a fixed $d$, there is a unique pair
$(\bar{k},\bar{l})$ satisfying both  twisting conditions, and Eq.
\cref{eq:od2} gives: \eq{\label{eq:sdpc}
\bbF_{L,\X|{z=\frac{(n+2)\a_1}{d}}}^{\rm
 disk}(\mathbf{1}_{\frac{\bar{k}}{n+2}})=
\pm \bbF_{L',Y|{z=\frac{(n+2)\a_1}{d}}}^{\rm
 disk}(\mathbf{1}_{\frac{\bar{l}}{n}}).
}
Disk invariants of winding $d$  for $\cX$ are  the coefficients of
the class $\mathbf{1}^{\frac{\bar{k}}{n+2}}$  after specializing
$z=\frac{(n+2)\a_1}{d}$ in  $\bbF_{L,\X}^{\rm
 disk}$, whereas for $Y$ they are obtained as the coefficients of the class $\mathbf{1}^{\frac{\bar{l}}{n}}$ after the same specialization of $z$ in  $\bbF_{L,Y}^{\rm
 disk}$. Hence,
summing over all $d$, Eq. \cref{eq:sdpc} yields the equality of
scalar potentials as stated in Theorem \ref{thm:speclegodd}.
\hfill $\Box$ \end{proof}

\subsubsection{Scalar disk potentials for the special leg: $n$ even}
\begin{thm}
Let $n$ be an even integer. Consider a Lagrangian boundary condition
$L$ on $\cX$ which intersects the first coordinate axis, and denote
by $L'$ the proper transform in $Y$. Then, upon identifying the
insertion variables via the change of variable prescribed by the
closed CRC, we have the equality of scalar disk potentials:
\eq{ F_{L',Y}^{\rm disk}(\tau,
y,\vec{w}) =  F_{L,\X}^{\rm
  disk}(\tau,
y, \vec{w}). } \label{thm:speclege}
\end{thm}

\begin{proof}
The transformation $\bbO$ in this case is the same as in
Section \ref{sec:nodd}. However we specialize  to $z=\frac{(n+2)\a_1}{2d}$
to obtain $\bbO_d$: \eq{ \bbO^k_{d,l} =
 \frac{
 \sin\left(\pi\left( \frac{2d(\a_1+\a_2)}{(n+2)\a_1} + \left\langle\frac{k}{n+2} \right\rangle\right) \right)
} {\sin\left(\pi\left( \frac{2d(\a_1+\a_2)}{(n+2)\a_1}+
\frac{2d}{n(n+2)}+ \left\langle -\frac{l}{n} \right\rangle\right)
 \right)}
 \frac{1}{n+2}
 {\rm e}^{\frac{2\pi {\rm i} j}{n(n+2)}\left(
 kn +l(n+2) -2d
 \right)}.
 \label{eq:od3}}
The expression in Eq. \cref{eq:od3} is summed over the index $j$
ranging from $0$ to $n+1$.  The degree-twisting compatibilities are:
\begin{itemize}
    \item[$\X$:] $ 2d\equiv kn \ \ \ \mbox{mod}\ \  n+2 ,$
    \item[$Y$:] $ 2d\equiv 2l \ \ \ \mbox{mod} \ \ n.$
\end{itemize}
Modular arithmetic again tells us that for any $d$ there are four
pairs of solutions to the above system of congruences, corresponding
to the solutions to: \eq{ 2d\equiv kn +l(n+2) \ \ \ \mbox{mod} \ \
\frac{ n(n+2)}{2}. \label{crt2} } Note that if $(k_0,l_0)$ is a
solution of (\ref{crt2}), then the other solutions are
$(k_0,l_1),(k_1,l_0), (k_1,l_1)$, where $k_1=k_0+\frac{n+2}{2}$ and
$l_1=l_0+\frac{n}{2}$. Without loss of generality we denote
$(k_0,l_0)$ and $(k_1,l_1)$  the solutions such that $2d\equiv kn
+l(n+2) \ \ \ \mbox{mod} \ \  {n(n+2)}$ and we observe that
$\bbO^{k_0}_{d,{l_0}}=\bbO^{k_1}_{d,{l_1}}=\pm1$, whereas
$\bbO^{k_0}_{d,{l_1}}=\bbO^{k_1}_{d,{l_0}}=0$. \\

Just as before, for $l=l_0, l_1$ and all other $k$'s, the
corresponding coefficients in the matrix $\bbO_d$ vanish. This gives
the equalities:
\eq{
\bbF_{L,\X|{z=\frac{(n+2)\a_1}{2d}}}^{\rm
disk}(\mathbf{1}_{\frac{k_0}{n+2}}) =\label{eq:sdpce}
 \pm \bbF_{L',Y|{z=\frac{(n+2)\a_1}{2d}}}^{\rm disk}(\mathbf{1}_\frac{l_0}{n}),
} \eq{
\bbF_{L,\X|{z=\frac{(n+2)\a_1}{2d}}}^{\rm
disk}(\mathbf{1}_{\frac{k_1}{n+2}}) =\label{eq:sdpce2}
 \pm \bbF_{L',Y|{z=\frac{(n+2)\a_1}{2d}}}^{\rm disk}(\mathbf{1}_\frac{l_1}{n})
. }
We recognize the disk invariants of winding $d$ for $\X$ (resp. $Y$)
in the sum of the left hand sides (resp. right hand sides) of Eq.
\cref{eq:sdpce} and
Eq. \cref{eq:sdpce2}.  Hence, 
summing over all $d$, Eq. \cref{eq:sdpc} yields the equality of
scalar potentials as stated in Theorem \ref{thm:speclege}.
\hfill $\Box$ \end{proof}

\section{Example 2: the closed topological vertex}
\label{sec:ctv}

\subsection{Classical geometry}


The closed topological vertex arises from the GIT quotient
construction \cite{MR1677117}
\eq{\begin{xy} (15,0)*+{0}="z";
(30,0)*+{\bbZ^{3}}="y";(50,0)*+{\bbZ^{6}}="a"; (70,0)*+{\bbZ^3}="b";
(85,0)*+{0}="c"; {\ar^{}  "z";  "y"}; {\ar^{M^T}  "y";"a"}; {\ar^{N}
"a";"b"};
{\ar^{} "b";"c"};
\end{xy},
\label{eq:GIT} }
where
\ea{ M= & \l( \bary{cccccc}
1 & 1 & 0&-2 & 0& 0  \\
1 & 0 & 1& 0 &-2 & 0   \\
 0 & 1 & 1 & 0 & 0 & -2 \\
\eary
 \r),
\label{eq:ctvM} &
 N= & \l(
\bary{cccccc}
0 & 2 & 0 &1& 0 & 1   \\
 0 & 0 & 2 & 0 & 1 & 1   \\
1 & 1 & 1 & 1 & 1 & 1  \\
\eary \r). }
The resulting geometry is a quotient $
\bbC^6/\!\!/_\chi(\bbC^\star)^3$, where the characters of the torus
action on the affine coordinates $x_1,\dots,
x_6$ of $\bbC^6$ are encoded in the rows of $M$. \\

\begin{figure}[h]
\centerline{\includegraphics{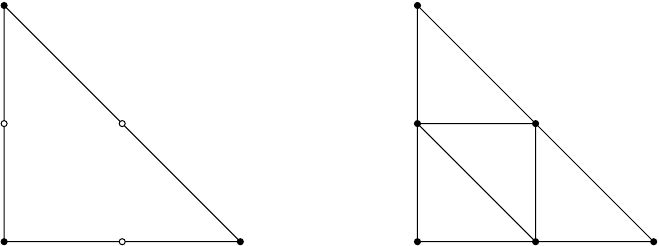}}
\caption{Fans of $[\bbC^3/(\bbZ_2 \times \bbZ_2)]$ (left) and its
  $G$-Hilb canonical resolution (right), depicting a slice of the three dimensional picture with a horizontal hyperplane at height $1$.}
\label{fig:fanctv}
\end{figure}

In two distinct chambers, the GIT quotient yields the toric
varieties whose fans are given by cones over Figure \ref{fig:fanctv}.
The picture on the left hand side corresponds to the orbifold
chamber: we delete the unstable locus
%
%
%
\eq{ \Delta_{\rm OP} \triangleq V\l( \bra x_4x_5x_6\ket \r). }
and then quotient by Eq. \cref{eq:ctvM}:
using the torus action to make $x_4$, $x_5$ and $x_6$  equal to $1$
gives a residual effective $\mu_2^3/\mu_2\cong \bbZ_2 \times \bbZ_2$
action\footnote{We choose the isomorphism given by $(0,1)$ being the
element whose representation fixes $z$, $(1,0)$ fixing $y$ and
$(1,1)$ fixing~$x$.} on $\bbC^3$ with coordinates $x_1$, $x_2$,
$x_3$. We  denote by $\cX \triangleq [\bbC^3/(\bbZ_2 \times
\bbZ_2)]$ the resulting orbifold, and by $X$ its coarse
moduli space. \\

The picture on the right hand side corresponds instead to the
distinguished large radius chamber that gives rise to Nakamura's
Hilbert scheme of $(\bbZ_2 \times \bbZ_2)$-clusters: we delete the
set
%
%
%
\eq{ \Delta_{\rm LR} \triangleq V\l(\prod_{(i,j,k) \neq (1,4,5),
(2,4,6), (3,5,6), (4,5,6)}\bra x_i,x_j,x_k\ket\r) }
and then quotient by the $(\bbC^\star)^3$ action in Eq.
\cref{eq:ctvM}; we will denote by $Y$ the corresponding smooth toric
variety. This is the trivalent geometry on the right-hand-side of
Figure \ref{fig:toricctv}: the local geometry of three $(-1,-1)$ curves
inside a Calabi--Yau threefold intersecting at a point.

\begin{figure}[h]
\centerline{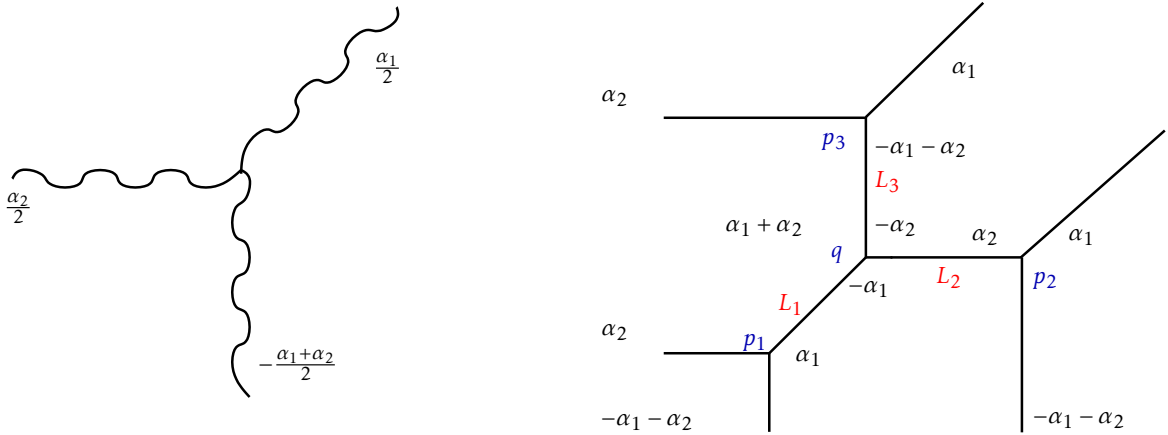}
\caption{Toric web diagrams and weights at the fixed points of
$[\bbC^3/(\bbZ_2 \times \bbZ_2)]$ (left) and its
  $G$-Hilb canonical resolution (right).}
\label{fig:toricctv}
\end{figure}

\subsubsection{Bases for cohomology}
\label{sec:cec} We equip  $Y$ and $\cX$ with a Calabi--Yau 2-torus
action descending from the action on $\bbC^6$ with geometric weights
$(\a_1,\a_2, -\a_1-\a_2,0,0,0)$.
%
%
This descends to give an effective $T\simeq (\bbC^*)^2$ action on
$Y$ and $\cX$ which preserves their canonical bundle; the resolution
diagram
\eq{
\begin{xy}
(0,25)*+{Y}="a"; (50,25)*+{\cX}="b";(25,0)*+{X}="c"; {\ar^{\rho}
"a";  "c"}; {\ar_{\pi}  "b";"c"};
\end{xy}
\label{eq:resYX} }
is $T$-equivariant.\\

Bases for the equivariant cohomology of $Y$ and $\cX$ can be
presented as follows. Let $L_i \subset Y$, $i=1,2,3$ denote the
torus-invariant projective lines
%
\ea{ \label{eq:L1}
L_1 =& V(x_4, x_5), \\
\label{eq:L2}
L_2  = & V(x_4, x_6), \\
\label{eq:L3} L_3 = & V(x_5, x_6). }
The cohomology of $Y$ is generated as a module by the duals
$\omega_i=[L_i]^\vee \in H_2(Y)$ of the fundamental classes in Eqs.
\cref{eq:L1}--\cref{eq:L3}, plus the identity
class $\mathbf{1}_Y\in H_0(Y)$.
The action on $\bbC^6$  above yields canonical lifts of
$i_{L_j}^*\omega_i=c_1(\cO_{L_j}(\delta_{ij}))$ to equivariant
cohomology. Denoting by $q$ the intersection of the three fixed
lines, $p_i$ the other torus fixed point of $L_i$, and by capital
letters the corresponding cohomology classes,
 the Atiyah--Bott isomorphism sends:
\ea{ \label{eq:abctv1}
\omega_1 & \to  \frac{\a_1}{2} (Q-P_1+P_2+P_3), \\
\label{eq:abctv2}
\omega_2 & \to \frac{\a_2}{2}  (Q+P_1-P_2+P_3), \\
\label{eq:abctv3} \omega_3 & \to -\frac{\a_1+\a_2}{2}
(Q+P_1+P_2-P_3). }
%
%
The $T$-equivariant Poincar\'e pairing
$\eta^Y(\phi_1,\phi_2)=\sum_{P_i}\phi_1|_{P_i}
\phi_2|_{P_i}\re^{-1}(N_{P_i/Y})$, in the basis $(Q,P_1,P_2,P_3)$
for $H^\bullet_T(Y)$, takes the block-diagonal form
\eq{ \eta^Y= \left(
\begin{array}{cccc}
 \frac{2}{\a_2 \a_1^2+\a_2^2 \a_1} & 0 & 0 & 0 \\
 0 & \frac{\a_1}{2 \a_2^2+2 \a_1 \a_2} & -\frac{1}{2 (\a_1+\a_2)} & \frac{1}{2 \a_2} \\
 0 & -\frac{1}{2 (\a_1+\a_2)} & \frac{\a_2}{2 \a_1^2+2 \a_2 \a_1} & \frac{1}{2 \a_1} \\
 0 & \frac{1}{2 \a_2} & \frac{1}{2 \a_1} & \frac{1}{2} \left(\frac{1}{\a_2}+\frac{1}{\a_1}\right) \\
\end{array}
\right). }
On $\cX$, the torus equivariant cohomology is spanned by the
$T$-equivariant cohomology classes $\mathbf{1}_g$, labeled by the
corresponding group elements $g=(0,0)$, $(0,1)$, $(1,0)$ and $(1,1)$. \\

\subsubsection{The grade restriction window}
\label{sec:gradectv} Consider the natural restriction window
$\mathfrak{W}$ consisting of the trivial representation of
$(\bbC^\ast)^3$ and the three one dimensional representations whose
characters are given by  the first three columns of the matrix $M$
in Eq. \cref{eq:ctvM}.  These descend to the four irreducible
representations of $\cX$, whose nontrivial characters  are still
encoded by the first three columns of $M$ via
$\rm{i}\pi$-exponentiation; and to the bundles $\cO$ and
$\cO_{L_j}(\delta_{ij})$ on $Y$. Using $\mathfrak{W}$ to identify
the $K$-lattices, the natural basis of irreducible representations
for $H^\bullet_{T}(\cX)$ and the fixed point basis for
$H^\bullet_{T}(Y)$, the matrix  representing the (homogenized,
involution pulled-back) Chern character for $\cX$ and $Y$ are
\ea{ ({\overline{\rm CH}_{\cX}})_j^k \triangleq  & {\left(\frac{2\pi
\rm{i}}{z}\right)}^{\frac{1}{2}\deg }\mathrm{inv}^\ast
\mathrm{CH}_\cX= \left(
\begin{array}{cccc}
1 &1&1& 1\\
1 &-1&-1& 1\\
1 &-1&1& -1\\
1 &1&-1&-1
\end{array}
\right)
\label{eq:chXctv} \\
({\overline{\rm CH}_{Y}})_j^l=& \left(
\begin{array}{cccc}
1 &\re^{\frac{\pi \mathrm{i} \a_1}{z}}&\re^{\frac{\pi \mathrm{i} \a_2}{z}}& \re^{-\frac{\pi \mathrm{i} (\a_1+\a_2)}{z}}\\
1 &\re^{-\frac{\pi \mathrm{i} \a_1}{z}}&\re^{\frac{\pi \mathrm{i} \a_2}{z}}& \re^{-\frac{\pi \mathrm{i} (\a_1+\a_2)}{z}}\\
1 &\re^{\frac{\pi \mathrm{i} \a_1}{z}}&\re^{-\frac{\pi \mathrm{i} \a_2}{z}}& \re^{-\frac{\pi \mathrm{i} (\a_1+\a_2)}{z}}\\
1 &\re^{\frac{\pi \mathrm{i} \a_1}{z}}&\re^{\frac{\pi \mathrm{i} \a_2}{z}}&
\re^{\frac{\pi \mathrm{i} (\a_1+\a_2)}{z}}
\end{array}
\right). \label{eq:chYctv}
}

\subsection{Quantum geometry}

The primary $T$-equivariant Gromov--Witten invariants of $Y$ were
computed for all genera and degrees in \cite{MR2207792}. Let $d_i$,
$i=1,2,3$ be the degrees of the image of a stable map to $Y$
measured with respect to the basis $L_i$, $i=1,2,3$ of
$H_2(Y,\bbZ)$, and suppose that $d_1+d_2+d_3 \neq 0$. Then
\cite[Prop.~11--15]{MR2207792}
\eq{ \int_{\overline{\cM_{g,0}}(Y; d_1,d_2,d_3)} 1 =
\frac{|B_{2g}|(2g-1)}{(2g)! (d_1+d_2+d_3)^{3-2g}}
\begin{cases}
1 & d_1=d_2=d_3, \\
1 & d_i=d_j=0, d_k>0, i\neq j\neq k, \\
-1 & d_1=d_j>0, d_k=0, i\neq j\neq k, \\
0 & \mathrm{else.}
\end{cases}
\label{eq:gwctv} }
The genus-zero Gromov--Witten potential then takes the form
\ea{ & F^Y(t)   \triangleq  \frac{1}{3!}\eta^Y(\phi, \phi \cup \phi)
+ \sum_{n \geq
    0} \sum_{d_1,d_2,d_3}\int_{\overline{\cM_{0,n}}(Y;d_1,d_2,d_3)} \frac{\prod_{i=1}^n \ev^*_i \phi}{n!}
\nn \\
& =  \frac{1}{6} \left(\frac{t_0^3}{\alpha _1 \left(-\alpha
_1-\alpha _2\right)
   \alpha _2}+\frac{\left(t_0-\alpha _2 t_2\right){}^3}{\alpha _1 \alpha _2 \left(\alpha _1+\alpha _2\right)}+\frac{\left(\left(\alpha _1+\alpha _2\right)
   t_3+t_0\right){}^3}{\alpha _1 \alpha _2 \left(\alpha _1+\alpha _2\right)}+\frac{\left(t_0-\alpha _1 t_1\right){}^3}{\alpha _1 \alpha _2 \left(\alpha _1+\alpha
   _2\right)}\right) \nn \\ & +  \text{Li}_3\left(\re^{t_1}\right)+\text{Li}_3\left(\re^{t_2}\right)-\text{Li}_3\left(\re^{t_1+t_2}\right)+\text{Li}_3\left(\re^{t_3}\right)-\text{Li}_3\left(\re^{t_1+t_3}\right)-\text{Li}_3\left(\re^{t_2+t_3}\right)+\text{Li}_3\left(\re^{t_1+t_2+t_3}\right)
\label{eq:prepctv} }
where we denoted $H_T(Y) \ni \phi := \sum_{i=0}^3 t_i \omega_i$ and
$\Li_3(x)$ is the polylogarithm function of order 3:
\eq{ \Li_n(y) = \sum_{k>0}\frac{y^k}{k^n}. }

As far as $\cX$ is concerned, its quantum cohomology was determined
in \cite{MR2518631} by an explicit calculation of $\bbZ_2 \times
\bbZ_2$ Hurwitz--Hodge integrals. Introduce linear coordinates
$x_{i,j}$ on the $T$-equivariant Chen--Ruan cohomology of $\cX$ by
$H^{\rm orb}_T(\cX) \ni\varphi := \sum_{i,j\in
  {0,1}} x_{i,j}\mathbf{1}_{(i,j)}$. Then \cite[Thm.~ 2]{MR2518631},
\eq{ F^\cX(x) = F^Y(t(x)) }
where the Bryan--Graber change of variables $t(x)$ reads
\eq{ \l( \bary{c}
t_0 \\
t_1 \\
t_2 \\
t_3 \\
\eary \r) = \left(
\begin{array}{cccc}
 1 & \frac{1}{2} \ri \a_1 & \frac{1}{2} \ri \a_2 &
   -\frac{1}{2} \ri (\a_1+\a_2) \\
 0 & \frac{\ri}{2} & -\frac{\ri}{2} & -\frac{\ri}{2} \\
 0 & -\frac{\ri}{2} & \frac{\ri}{2} & -\frac{\ri}{2} \\
 0 & -\frac{\ri}{2} & -\frac{\ri}{2} & \frac{\ri}{2} \\
\end{array}
\right) \l( \bary{c}
x_{0,0} \\
x_{1,0} \\
x_{0,1} \\
x_{1,1} \\
\eary \r) +\frac{\ri \pi}{2} \l( \bary{c}
0 \\
1 \\
1 \\
1 \\
\eary \r) . }

\subsection{One-dimensional mirror symmetry}
In the analysis of the disk and quantized CRC for the type~A
resolutions in \cite{Brini:2013zsa}, a prominent role was played by
a realization of the $D$-modules underlying quantum cohomology in
terms of a single-field logarithmic Landau--Ginzburg model, or, in
the language of \cite{2012arXiv1210.2312R}, the Frobenius dual-type
structure on a genus-zero double Hurwitz space. This was motivated
by a connection of the Gromov--Witten theory for these targets with
a class of reductions of the 2-Toda hierarchy \cite{Brini:2014mha}.
A similar connection with integrable systems holds for the closed
topological vertex as well; the general story will appear elsewhere,
but its consequences for the purposes of the paper are discussed below. \\

Define
\ea{ Z_1 \triangleq & -\frac{\re^{\frac{t_2}{2}}
\left(\re^{t_1}-1\right)
   \left(\re^{t_3}-1\right)}{\left(\re^{t_1+t_2}-1\right){}^2}, &
Z_2 \triangleq \frac{\re^{-\frac{t_2}{2}} \left(\re^{t_2}-1\right)
   \left(\re^{t_1+t_2+t_3}-1\right)}{\left(\re^{t_1+t_2}-1\right){}^2}, \nn \\
Z_3 \triangleq & \frac{\re^{t_1+\frac{t_2}{2}}
\left(\re^{t_2}-1\right)
   \left(\re^{t_3}-1\right)}{\left(\re^{t_1+t_2}-1\right){}^2}, &
Z_4 \triangleq -\frac{\re^{\frac{t_2}{2}} \left(\re^{t_1}-1\right)
   \left(\re^{t_1+t_2+t_3}-1\right)}{\left(\re^{t_1+t_2}-1\right){}^2}.
\label{eq:Zi} }
Fix now a branch $\cC$ of the logarithm  and denote by
$\cM_{\a_1,\a_2} \simeq M_{0,6} \times \bbC^*$ the smooth complex
four-dimensional manifold of multi-valued functions $\lambda(q)$ of
the form
\ea{ \cM_{\a_1,\a_2} = & \Big\{\lambda(q) =
t_0+\frac{(\a_1-\a_2)t_2}{2}
+\a_1 \log(Z_1-q)(Z_2-q) +\a_2 \log (Z_3-q)(Z_4-q) \nn \\
- & (\a_1+\a_2)\log q; \quad Z_i \neq 0,1, Z_j\Big\}.
\label{eq:Ma1a2} }
A given point $\lambda \in \cM_{\a_1,\a_2}$ is a perfect Morse
function in $q$ with four critical points $q^{\rm cr}_i$,
$i=1,\dots, 4$; its critical values,
\eq{ \label{eq:cancoord} u^i=\log\lambda(q^{\rm cr}_i), }
give a system of local coordinates on $\cM_{\a_1,\a_2}$, which is
canonical up to permutation.  Define now holomorphic tensors $\eta
\in \Gamma(\mathrm{Sym}^2 T^* \cM_{\a_1,\a_2})$, $c \in
\Gamma(\mathrm{Sym}^3 T^* \cM_{\a_1,\a_2})$ on $\cM_{\a_1,\a_2}$ via
\ea{ \label{eq:etaLG} \eta(\de, \de')= & \sum_{i=1}^4
\mathrm{Res}_{q=q^{\rm cr}_i} \frac{\de(\lambda)
  \de'(\lambda)}{\lambda'(q)}\psi(q) \rd q, \\
c(\de, \de', \de'')= & \sum_{i=1}^4 \mathrm{Res}_{q=q^{\rm cr}_i}
\frac{\de(\lambda)
  \de'(\lambda) \de''(\lambda)}{\lambda'(q)}\psi(q) \rd q
\label{eq:cLG} }
for holomorphic vector fields $\de$, $\de'$, $\de''$ on
$\cM_{\a_1,\a_2}$, where
\eq{
\psi(q)=\frac{1}{\a_2}\l[\frac{1}{q-Z_1}+\frac{1}{q-Z_2}-\frac{1}{q}\r].
}
Whenever $\eta$ is non-degenerate, this defines a commutative,
unital product $\de \circ \de'$ on $\Gamma(T\cM_{\a_1,\a_2})$ by
``raising the indices'': $\eta(\de,\de' \circ \de'')=c(\de, \de',
\de'')$.
\begin{thm}
Eqs. \cref{eq:etaLG} and \cref{eq:cLG} define a semi-simple
Frobenius manifold structure $\cF_{\a_1,\a_2}\triangleq
(\cM_{\a_1,\a_2}, \eta, \circ)$ on $\cM_{\a_1,\a_2}$ with
covariantly constant unit. Moreover,
\eq{ \cF_{\a_1,\a_2} = QH_T(Y) \simeq QH_T(\cX) }
\end{thm}
\begin{proof}
Associativity and semi-simplicity of the product follow immediately
from the fact that the canonical coordinate fields, $\de_{u^i}$, are
a basis of idempotents of Eq. \cref{eq:cLG}. A straightforward
computation of the residues in Eq. \cref{eq:etaLG} in the coordinate
chart $t_i$ shows that Eq. \cref{eq:etaLG} is a flat metric and the
variables $t_i$ are a flat coordinate system for $\eta$; similarly,
a direct evaluation of Eq. \cref{eq:cLG} shows that the algebra
admits a potential function, which coincides with Eq.
\cref{eq:prepctv}.
\hfill $\Box$ \end{proof}

\begin{cor}
Let $\nabla_X^{(z)} Y = \rd_X Y + z X \circ Y$ be the Dubrovin
connection on $\cF_{\a_1,\a_2}$. Then a system of flat coordinates
for $\nabla_X^{(z)}$ is given by the periods
\eq{ \Pi_i= \frac{z}{(1-\re^{2\pi \ri \a_1/z})(1-\re^{(-1)^i 2\pi\ri
(\a_1+\a_2)/z})}\int_{\gamma_i} \re^{\lambda/z}\psi(q) \rd q
\label{eq:Pi} }
where $\gamma_1=[C_{Z_1}, C_\infty]$, $\gamma_2=[C_{0}, C_{Z_2}]$,
$\gamma_3=[C_{Z_2},C_\infty]$, $\gamma_4=[C_0, C_{Z_1}]$ and we
denoted by $C_x$ a simple loop encircling counterclockwise the point
$q=x$.
\end{cor}

This is \cite[Prop.~5.2]{Brini:2013zsa}, where the superpotential
and primitive differential $\lambda$ and $\phi$ there are identified
respectively with $\re^{\lambda}$ and $\psi(q) \rd q$ here: the
contours $\gamma_i$ give a basis of the first homology of the
complex line twisted by a set of local coefficients given by the
algebraic monodromy of $\re^{\lambda/z}$ around the singular points
$Z_i$, $0$ and $\infty$. The reason behind this particular choice of
basis, as well as the normalization factor in front of the integral,
will be apparent in the course of the asymptotic analysis of
Section \ref{sec:asymp}.

\begin{rmk}{\rm
In the language of \cite{2012arXiv1210.2312R}, the Frobenius
manifold $\cF_{\a_1,\a_2}$ is the Frobenius dual-type structure on
the genus zero double Hurwitz space $H_{0,\kappa}$ with ramification
profile $\kappa=(\a_1,\a_1,\a_2,\a_2,-\a_1-\a_2,\a_1-\a_2)$, with
$\re^{\lambda}$ as its superpotential and the third kind
differential $\psi(q)\rd q$ as its primitive one-form; the integrals
Eq. \cref{eq:Pi} were called the {\it twisted periods} of
$\cF_{\a_1,\a_2}$ in \cite{Brini:2013zsa}. The corresponding
Principal Hierarchy \cite{Dubrovin:1994hc} is a four-component
reduction of the genus-zero Whitham hierarchy with three punctures
\cite{Krichever:1992qe}. The special case $\a_1=\a_2=\a$ is
particularly interesting, as in that case $\cF_{\a,\a}$ is the dual
(in the sense of Dubrovin \cite{MR2070050}) of a conformal charge
one Frobenius manifold with non-covariantly constant identity; flat
coordinates for the two Frobenius structures are in bijection with
Darboux coordinates for a pair of compatible Poisson brackets for
the Principal Hierarchy, which thus gives rise to a (new)
bihamiltonian integrable system of independent interest. We will
report on it in a forthcoming work.}
\end{rmk}

\subsubsection{Computing $\U$}
Encoding the coefficients of $\overline{\Gamma}_\cX(z)$ and
$\overline{\Gamma}_Y(z)$  as entries of diagonal matrices, the
prediction for the symplectomorphism $\U$ from Iritani's theory of
integral structure is obtained by composing \eq{ \U=
\overline{\Gamma}_Y \circ  \overline{\rm CH}_{Y} \circ \overline{\rm
CH}_{\cX}^{-1} \circ \overline{\Gamma}_\cX^{-1}, } as we now turn to
verify. Let $\cY_\epsilon$ be the ball of radius $\epsilon$ around
$\re^{t}=0$, measured w.r.t. the Euclidean metric $(\rd s^2)=\sum_i
(\rd \re^{t_i})^2$ in exponentiated flat coordinates, and define the
path in $\cY_1$
\eq{ \bary{cccc}
\rho: & [0,1] & \to & \cY_1, \\
& y & \to & (\rho(y))_j= \ri y. \label{eq:rhoctv} \eary }
Beside $\Pi_i$, systems of flat coordinates for the deformed flat
connection $\nabla^{(z)}$ are given by the components of the
$J$-functions of $\cX$ and $Y$; the discrepancy between them encodes
the morphism of Givental spaces that identifies the Lagrangian cones
of $\cX$ and $Y$ under analytic continuation along the  path $\rho$:
\eq{ J^{Y} = \U J^\cX. }
As in \cite{Brini:2013zsa}, $\U$ can be computed in two steps, by
expressing $J^\bullet$ in terms of the periods $\Pi$,
\ea{ \label{eq:B}
\Pi_i = & \sum_{\a=0}^3 \cB_{i\a} J_{\a}^\cX, \\
\Pi_i = & \sum_{j=1}^r \cA^{-1}_{ij} J_j^Y, }
where $J^\cX_{\a}$ and $J^Y_j$ are the components of the
$J$-functions of $\cX$ and $Y$ respectively in the inertia basis of
$\cX$ and in the localized basis of $Y$; we have labeled elements of
$\bbZ_2\times \bbZ_2$ by a single index $\a=0,1,2,3$ for $g=(0,0)$,
$(1,0)$, $(0,1)$ and $(1,1)$ respectively. Throughout the rest of
this Section, in order to simplify formulas, we define
$\mu_i\triangleq\alpha_i/z$.
\begin{prop}
\label{prop:AB} We have
\eq{ \label{eq:A} \cA^{-1}\DD_0^{-1} = \left(
\begin{array}{cccc}
-\re^{\ri \pi  \left(\mu _1+\mu _2\right)}  \frac{\sin \left(\pi
\mu _1\right)}{\sin \left(\pi  \mu _2\right)} & 0 &  \re^{\ri \pi
\left(\mu _1+\mu _2\right)}\frac{\sin \left(\pi  \mu _1\right)}{
\sin \left(\pi  \mu _2\right)} &
   -1 \\
 -(-1)^{\mu _1}  \frac{\sin \left(\pi (\mu _1+\mu_2)\right)}{\sin \left(\pi  \mu _2\right)}
   &  (-1)^{2 \mu _1} & (-1)^{\mu _1} \frac{\sin \left(\pi  \left(\mu _1+\mu
    _2\right)\right)}{ \sin \left(\pi  \mu _1\right)} & 0 \\
 0 & 0 & -1 & 0 \\
 -1 & 1 & 0 & 1 \\
\end{array}
\right) }
\eq{ \cB_{i\a}= (\DD_1 \cI \DD_2)_{i\a} \label{eq:Bfin} }
where \eq{ \label{eq:DD0} \DD_0 =  \diag
\mu_2^{-1}\left(-B(\mu_1,-\mu_1-\mu_2),B(-\mu_1,\mu_1+\mu_2),
B(-\mu_1-\mu_2,1+\mu_2), -B(\mu_1,-\mu_1-\mu_2) \right), } \eq{
\DD_1 =  \diag \l(\re^{\frac{1}{2} \ri \pi  (2 \mu_1+\mu_2)},
\re^{\frac{1}{2} \ri \pi  (2
  \mu_1+3\mu_2)}, \re^{-\frac{1}{2} \ri \pi  \mu_2} , \re^{\frac{1}{2} \ri \pi
  \mu_2}\r),
\label{eq:DD1} } \ea{ \DD_2 = \diag\bigg[ & -\frac{2}{\mu_2} B
\left(\frac{\mu_1}{2},-\frac{\mu_1+\mu_2}{2}\right),\ri
B\left(\frac{\mu_1}{2},\frac{1}{2} (1-\mu_1-\mu_2)\right), \nn \\ &
- B
  \left(\frac{1}{2} (\mu_1+1),\frac{1}{2}
  (1-\mu_1-\mu_2)\right),\ri B \left(\frac{1}{2} (\mu_1+1),-\frac{\mu_1+\mu_2}{2}\right)\bigg],
\label{eq:DD2} } \eq{ \cI  =
 \frac{1}{4}
\left(
\begin{array}{cccc}
 -1 & -1 & 1 & 1 \\
 1 & -1 & -1 & 1 \\
 -1 & 1 & -1 & 1 \\
 1 & 1 & 1 & 1 \\
\end{array}
\right), \label{eq:I} }
and $B(x,y)$ denotes Euler's $\beta$-function
\eq{ B(x,y)=\frac{\Gamma(x) \Gamma(y)}{\Gamma(x+y)} }
\end{prop}

\begin{proof}
$J_{\a}^\cX(x,z)$ is characterized as the unique system of flat
  coordinates of $\nabla^{(z)}$ which is linear with no inhomogeneous term in
  $\re^{x_0/z}$ and satisfies
\eq{ \de_\a J_\b(0,z) = \delta_{\a,\b} }
at the orbifold point $x=0$. Then,
\eq{ \cB_{i,\a}= \de_\a \Pi_i(0,z). \label{eq:B-1} }
The integrals appearing on the r.h.s. of Eq. \cref{eq:B-1} can be
explicitly evaluated in terms of the Euler $\beta$-integral; this is
illustrated in detail in Appendix \hyperref[sec:B-1]{A.A}, and returns Eqs.
\cref{eq:DD1}--\cref{eq:I}. Similarly,
$J_{j}^Y(t,z)$ is characterized as the unique system of flat
  coordinates of $\nabla^{(z)}$ (linear with vanishing inhomogeneous term in
  $\re^{t_0/z}$) that diagonalizes the monodromy of $\nabla^{(z)}$ at large
  radius as
\ea{ J_j^Y(t,z) P_j = & z \l(i^*_{p_j} \re^{t \cdot
  \omega/z}\r)\l(1+\cO(\re^{t})\r) \nn \\
\sim & z \re^{t_0/z}
\begin{cases}
\re^{-\mu_1 t_1}P_1 & j=1, \\
 Q & j=2, \\
 \re^{-\mu_2 t_2}P_2 & j=3, \\
 \re^{(\mu_1+\mu_2) t_3}P_3 & j=4,
\end{cases}
\label{eq:JYas} }
where the r.h.s. is determined by the localization of $\omega_i$ at
$p_j$ as in Eqs. \cref{eq:abctv1}--\cref{eq:abctv3}. Then $\cA$ is determined by the decomposition of
the periods in terms of eigenvectors of the monodromy at large
radius, that is, by their asymptotic behavior as $\Re
(t)\to-\infty$. The details of the large radius asymptotics of
$\Pi_i$ are quite involved and are deferred to Appendix \hyperref[sec:A]{A.B};
the final result is Eq. \cref{eq:A}.
\hfill $\Box$ \end{proof}

\begin{cor}
Conjecture \ref{conj:iri} holds for $\cX=[\bbC^3/\bbZ_2 \times
\bbZ_2]$ and $Y\to X$ its $G$-Hilb resolution with grade restriction
window $\mathfrak{W}$ and analytic continuation path $\rho$ as in
Eqs. \cref{eq:chXctv}, \cref{eq:chYctv}, and \cref{eq:rhoctv}.
\end{cor}

\subsection{Quantization and the all-genus CRC}

For $j=1,\dots, 4$, define $1$-forms formally analytic in $z$,
$\RR_j = R_{ij}(u,z) \re^{u^j/z} \rd u^i$, satisfying the following
set of conditions:

\begin{itemize}
\item[\bf R1:] $R_{ij}(u,z) \in \cO_{\cM_{\a_1,\a_2}} \otimes \bbC[[z]]$,
\item[\bf R2:] $\nabla^{(z)} \RR_j = 0$ as a formal Taylor series in $z$,
\item[\bf R3:] $\sum_j R_{ij}(u,z) R_{kj}(u,-z) = \delta_{ik}$.\\
\end{itemize}

By condition {\bf R2} and their prescribed singular behavior at
$z=0$, $\RR_j$ are formal (asymptotic) flat sections of the Dubrovin
connection uniquely defined up to right multiplication by constants,
$R_{ij}(u,z) \to R_{ij}(u,z) \cN_j(z)$; picking a choice of $\RR$ is
said to endow $\cF_{\a_1,\a_2}$ with an {\it $R$-calibration}. Write
$B_{k}$ for the $k^{\rm th}$ Bernoulli number,
\eq{ \sum_{k\geq 0} B_k \frac{t^k}{k!}\triangleq
\frac{t}{\re^{t}-1}, }
and let $\Delta_i(u)$ be the normalized inverse-square-length of the
coordinate vector field $\de_{u^i}$ in the Frobenius metric, Eq.
\cref{eq:etaLG}. We will also denote by $\psi^\cW$ the Jacobian
matrix of the change-of-variables from the canonical frame, Eq.
\cref{eq:cancoord}, to the flat coordinate systems $\mathbf{t}$ and
$\mathbf{x}$ for $\cW=Y$ and $\cX$ respectively, with columns
normalized by $\sqrt{\Delta}$.
\begin{defn}
{\rm
The {\rm Gromov--Witten $R$-calibration} $(\RR_Y)_j =
(R_Y)_{ij}(u,z) \re^{u^j/z} \rd u^i$ of $Y$ is the unique
$R$-calibration on $QH_T(Y)\simeq \cF_{\a_1,\a_2}$ such that
\eq{ \lim_{\Re (t) \to -\infty}(R_Y)_{ij}(u,z) = \mathscr{D}^Y_i(z)
\delta_{ij}, \label{eq:normRY} }
where \eq{ \mathscr{D}^Y_i(z) =
\begin{cases}
\exp\l[\sum_{k>0}\frac{B_{2k}}{2k(2k-1)}
  \l(-\mu_1^{1-2k}-\mu_2^{1-2k}+(\mu_1+\mu_2)^{1-2k}\r) \r] & i=1,\\

\exp\l[\sum_{k>0}\frac{B_{2k}}{2k(2k-1)}
  \l(\mu_1^{1-2k}+\mu_2^{1-2k}-(\mu_1+\mu_2)^{1-2k}\r) \r] & \mathrm{else.}
\end{cases}
}
The \, {\rm Gromov--Witten \, $R$-calibration} \, $(\RR_\cX)_j\,=\,(R_\cX)_{ij}\,(u,z)\,\re^{u^j/z}\,\rd u^i$ \,of \,$\cX$ \,is \,the \,unique
\,$R$-calibration \,on $QH_T(\cX)\simeq \cF_{\a_1,\a_2}$ satisfying
\eq{ \sum_{i}\psi^\cX_{\a i} R^\cX_{ij}(u,z)\Big|_{x=0} = \l(e^{\rm
eq}(V^{(0)})\r)^{-1/2}\mathscr{D}^\cX_\a(z) \chi_{\a j},
\label{eq:normRX} }
where $\chi_{\a j}$ is the character table of $\bbZ_2 \times
\bbZ_2$, $V^{(0)}$ is the trivial part of the representation $V$
(thought of as a vector bundle on the classifying stack), and \eq{
\label{eq:DX} \mathscr{D}^\cX_\a  =
\begin{cases}
z\exp\left[\sum_{k>0}\frac{B_{2k} z^{2k-1}}{2k(2k-1)}\l(
\mu_1^{1-2k}+\mu_2^{1-2k}-(\mu_1+\mu_2)^{1-2k}\r)\r]
& \a=0,\\
\frac{\ri}{\sqrt{\mu_2 (\mu_1+\mu_2)}}
                \exp \left[\sum_{k>
                  0}\frac{B_{2k}z^{2k-1}}{2k(2k-1)}\left(\frac{1}{
                  \mu_1^{2k-1}}+\frac{2^{1-2k}-1}{ \mu_2^{2k-1}}+\frac{1-2^{1-2k}}{
                  (\mu_1+\mu_2)^{2k-1}}\right)\right] & \a=1, \\
       -\frac{1}{\sqrt{-\mu_1 (\mu_1+\mu_2)}}

\exp \left[\sum_{k>
                  0}\frac{B_{2k}z^{2k-1}}{2k(2k-1)}\left(\frac{1}{
                  \mu_2^{2k-1}}+\frac{2^{1-2k}-1}{ \mu_1^{2k-1}}+\frac{1-2^{1-2k}}{
                  (\mu_1+\mu_2)^{2k-1}}\right)\right] & \a=2, \\
\frac{\ri}{\sqrt{-\mu_1\mu_2}}

\exp \left[\sum_{k>
                  0}\frac{B_{2k}z^{2k-1}}{2k(2k-1)}\left(\frac{2^{1-2k}-1}{
                  \mu_1^{2k-1}}+\frac{2^{1-2k}-1}{ \mu_2^{2k-1}}-\frac{1}{
                  (\mu_1+\mu_2)^{2k-1}}\right)\right] &  \a=3.
\end{cases}
} \label{def:R}}
\end{defn}

For either $\cX$ or $Y$, Eqs. \cref{eq:normRY} and \cref{eq:normRX}
together with conditions {\bf R1-R3} above determine the
Gromov--Witten $R$-calibration uniquely. Existence of an
$R$-calibration $\RR_Y$ compatible with Eq. \cref{eq:normRY} follows
from the general theory of semi-simple quantum cohomology of
manifolds; the existence of an asymptotic solution $\RR_\cX$ of the
deformed flatness equations satisfying the (a priori
over-constrained) normalization condition Eq. \cref{eq:normRX} will
be shown in
the course of the proof of Theorem \ref{thm:scr}. \\

The relevance of Definition \ref{def:R} is encoded in the following
statement, which condenses \cite[Thm.~9.1]{MR1901075} and
\cite[Lem.~6.3,~6.5]{Brini:2013zsa}.
\begin{prop}
Givental's quantization formula holds for $\cW=\cX$ or $Y$ in any
path-connected domain containing the large radius point of $\cW$,
\eq{ \label{givqf} Z^\cW(\mathrm{t}_u) = \widehat{S_\cW^{-1}}
\widehat{\psi_\cW} \widehat{R_\cW} \re^{\widehat{u/z}} \prod_{i=1}^4
Z_{i, \rm pt}. }
where $\mathrm{t}_u$ denotes the shifted descendent times
$\mathrm{t}^{p}_u=t^{p}+\tau_\cW(u) \delta_{p0}$. Moreover, the
Coates--Iritani--Tseng/Ruan quantized CRC,
\eq{ Z^Y(\mathrm{t}_u) = \widehat{\U} Z^\cX(\mathrm{t}_u), }
holds if and only if the Gromov--Witten $R$-calibrations agree on
the semi-simple locus,
\eq{ R^\cX(u,z) = R^Y(u,z). }
\end{prop}

\subsubsection{Saddle-point asymptotics}
\label{sec:asymp}

\begin{figure}
\centerline{\resizebox{.5\textwidth}{!}{
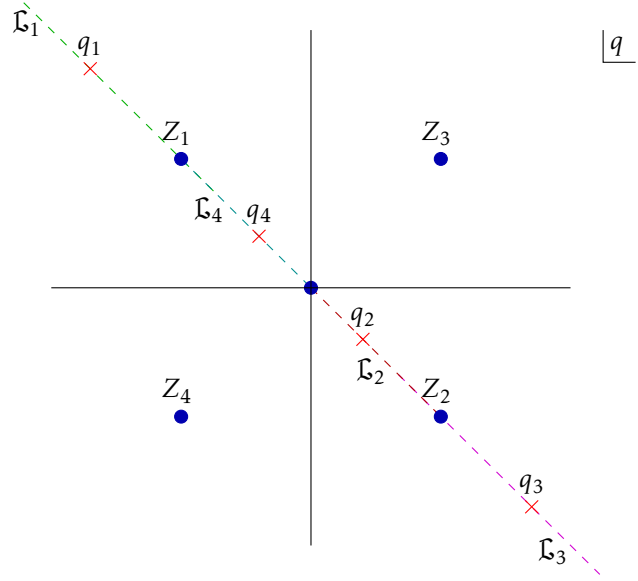
}} \caption{Singular and critical points of the superpotential at the
orbifold
  point. $Z_1$, $Z_2$ are negative log-infinities of the superpotential. $Z_3$ and $Z_4$ are
  positive log-infinities. $q_i$, $i=1,2,3,4$ are the critical points.}
\label{fig:critOP}
\end{figure}

Formal power series solutions in $z$ of $\nabla^{(z)}\RR =0$ are
obtained from the saddle-point asymptotics of Eq. \cref{eq:Pi} at
$z=0$. The latter is an essential singularity of the horizontal
sections of the Dubrovin connection, and their asymptotic analysis
at $z=0$ relies on a choice of phase for the parameters $\a_1$,
$\a_2$, $z$ -- namely, a choice of Stokes sector. A technically
convenient choice is to restrict our study to the wedge
%
$\cS_+ = \{(\mu_1,\mu_2)| \Re(\mu_1)>0, \Re(\mu_2)<-\Re(\mu_1)\}$;
%
as individual correlators depend rationally on $\mu_1$, $\mu_2$, our
statements  will hold in full generality by analytic continuation in
the space of
equivariant parameters.

\begin{thm}
The all-genus, full-descendent Crepant Resolution Conjecture
(Conjecture \ref{conj:scr}) holds with $\cX=[\bbC^3/\bbZ_2 \times
\bbZ_2]$, $Y\to X$ its $G$-Hilb resolution and $\rho$ the analytic
continuation path of Eq. \cref{eq:rhoctv}. \label{thm:scr}
\end{thm}

\begin{proof}
Asymptotic horizontal sections $\RR_i(u,z)$ are given by the
classical Laplace asymptotics of the integrals
\eq{ \cI_i = z \int_{\mathfrak{L}_i}\re^{\lambda/z} \phi(q) \rd q
\label{eq:Iint} }
where the Lefschetz thimble $\mathfrak{L}_i$ is given by the union
of the downward gradient lines of $\Re(\lambda)$ emerging from its
$i^{\rm th}$ critical point. Let us first consider the situation at
the orbifold point, which is schematized in Figure \ref{fig:critOP}.
We compute from Eq. \cref{eq:Ma1a2}
\eq{ q^{\rm cr}_i\Big|_{x=0}=   \frac{(-1)^{1/4+\sigma(i)}}{2}
\overline{q}^{(-1)^i}, \quad \overline{q} =   \sqrt{\frac{\sqrt{\mu
_1}+\sqrt{-\mu _2}}{\sqrt{\mu
      _1}-\sqrt{-\mu _2}}},
} \eq{ \l\{Z_1,Z_2,Z_3,Z_4\r\}\Big|_{x=0} =  \frac{\re^{\pi
\ri/4}}{2} \l\{\ri,-\ri,-1,1\r\}, }
with $\sigma(1)=\sigma(4)=0$, $\sigma(3)= \sigma(2)=1$. It is
straightforward to check that the constant phase paths of
$\re^{\lambda/z}$ emerging from $q^{\rm cr}_i$ are the straight
lines $\arg(q)=\pm \pi (\sigma(i)+1/4)$ that terminate at the
nearest algebraic zero of $\re^{\lambda/z}$ or at infinity, as in
Figure \ref{fig:critOP}. Moreover, for our choice of phases of the
weights in $\cS_+$, the contour integrals of $\re^{\lambda/z}\psi$
around the Pochhammer contours $\gamma_i$ retract
\cite[Rmk~5.5]{Brini:2013zsa} to line integrals on the straight line
segments connecting the zeroes of $\re^{\lambda/z}$ inside
$\gamma_i$. At the orbifold point, these are precisely the Lefschetz
thimbles $\mathfrak{L}_i$: then, the saddle-point expansion of the
differentials $\RR_i = \psi^\cX_{\a j} R_{ji}(u,z) \re^{u^i/z}\rd
x^\a \triangleq \rd \cI_i = \rd \Pi_i$ satisfies conditions {\bf
R1-R2 } above. We claim that up to right multiplication by $\cN_i
\in \bbC[[z]]$, $\RR_i$ this satisfies {\bf R3} and coincides with
the Gromov--Witten $R$-calibration of $\cX$. Indeed, as shown in
Appendix \hyperref[sec:B-1]{A.A}, in the trivialization given by $x^\a$ the
differential of the periods of $\re^{\lambda/z}$ at $x=0$ reduce to
Euler Beta integrals, whose steepest-descent asymptotics is
determined by Stirling's expansion for the $\Gamma$ function:
\eq{ \Gamma(x+y) x^{-x}\re^{x}x^{1/2-y} \stackrel{}{\simeq}
\sqrt{2\pi}\exp\l(\sum_{k>0}\frac{B_{k+1}(1-y)}{k(k+1)}x^{k}\r),
\quad  \Re (x)\gg 0. \label{eq:stirling} }
Then:
\ea{ \re^{-u^i/z}\de_{x^\a}\Pi_i\Big|_{x=0} = &
\re^{-u^i/z}|_{x=0}\cB^{-1}_{i\a} \nn \\ \simeq & \psi^\cX_{aj}
R_{ji}\Big|_{x=0} }
and by Eqs. \cref{eq:B-1}, \cref{eq:stirling}, and \cref{eq:DX} we
obtain
\eq{ \psi^\cX_{aj}R_{ji}\Big|_{x=0} = \sqrt{\frac{2\pi}{e^{\rm
eq}(V^{(0)})_\a}}\mathscr{D}^\cX_a \chi_{ai} }
so that $\RR = \sqrt{2\pi} \RR^\cX$. In particular, since by Eq.
\cref{eq:DX} $\RR$ satisfies the unitarity condition  at $x=0$, and
because parallel transport under the Dubrovin connection is an
isometry of the pairing in {\bf R3}, it satisfies condition {\bf R3}
for all $u$. At large radius, by condition {\bf R1} and the
asymptotic behavior of $J^Y(t,z)$ around $\Re(t)\to -\infty$ (Eq.
\cref{eq:JYas}), we must have that
\eq{ \RR \simeq \rd J^Y \cN^Y }
for some $\cN^Y =\lim_{\Re(t)\to-\infty}\re^{-u/z}\cI \in \bbC[[z]]
$. Its calculation via the steepest descent analysis of Eq.
\cref{eq:Iint} at large radius requires extra care since $\re^{t}=0$
is a singular point for $\nabla^{(z)}$: in this limit, the critical
points of the superpotential either coalesce at zero or drift off to
infinity,
\ea{ q^{\rm cr}_1 \sim & \frac{\a_1}{\a_2}\re^{t_2/2}, & q^{\rm
cr}_2 \sim & \l(1+
\frac{\a_1}{\a_2}\r)\re^{t_1+t_2/2}, \nn \\
q^{\rm cr}_3 \sim & \frac{\a_2}{\a_1+\a_2}\re^{-t_2/2}, & q^{\rm
cr}_4 \sim & -\re^{t_2/2}. }
The essential divergences in the saddle-point computation of $\cN^Y$
from Eq. \cref{eq:Iint} can be treated as follows: first rescale the
integration variables in Eq. \cref{eq:Iint} by $\re^{-t_2/2}$,
$\re^{-t_1-t_2/2}$, $\re^{t_2/2}$ and $\re^{-t_2/2}$ for $i=1,2,3,4$
respectively; then integrate over the steepest descent path,
isolating the essential divergence at the large radius point, and
finally take the resulting (finite) limit $\mathfrak{Re}(t)\to
-\infty$: notice that the last two steps do not commute in general,
as poles are generally created along the integration contour in the
large radius limit. The final result reduces, for all $i$, to the
computation of the saddle-point asymptotics of Beta integrals.
Explicitly, we get
\ea{ \sqrt{\Delta^{\rm cl}} \cN^Y=&
\lim_{\mathfrak{Re}(t)\to-\infty}\sqrt{\Delta_i(u)}\re^{-u_i/z}\cI_i
\nn \\ = &
\begin{cases}
\frac{2\pi \mu_1^{\mu_1-1/2}(-\mu_2)^{\mu_2+1/2}(-\mu_1-\mu_2)^{-\mu_1-\mu_2-1/2}}{B^{\rm as}(\mu_1,-\mu_1-\mu_2)} & i=1, \\
 \frac{B^{\rm as}(\mu_1,-\mu_2-\mu_1)}{\mu_1^{\mu_1-1/2}(-\mu_2)^{\mu_2+1/2}(-\mu_1-\mu_2)^{-\mu_1-\mu_2-1/2}} & \mathrm{else,}
\end{cases}
}
where $\Delta^{\rm cl}=\lim_{\Re(t)\to-\infty}\Delta(u)$ and $B^{\rm
as}(x,y)$ denotes the Stirling expansion of the Euler Beta function.
Then,
\eq{ \lim_{\Re(t)\to-\infty}R_{ij}(u,z)=\sqrt{2\pi}\mathscr{D}^Y_i
\delta_{ij}, }
and thus $R^\cX=R^Y$, concluding the proof.
\hfill $\Box$ \end{proof}

\begin{cor}
The quantized OCRC, Proposal 4, holds for $\cX$ and $Y$ as in
Theorem \ref{thm:scr}. \label{cor:scr}
\end{cor}

\appendix{}{Boundary behavior of periods}

For $|x_i|<1$, $i=1,2,3$, and $\Re(c)>\Re(a)>0$ let
$F_D^{(3)}(a,b_1,b_2,b_3,c,x_1,x_2,x_3)$ denote the Lauricella
hypergeometric function of type~$D$ \cite{MR0422713},
\ea{ \label{eq:FD} F_D^{(3)}(a,b_1,b_2,b_3,c,x_1,x_2,x_3)\triangleq
& \sum_{d_1,d_2,d_3\geq
  0}\frac{(a)_{d_1+d_2+d_3}}{(c)_{d_1+d_2+d_3}}\prod_{i=1}^3\frac{(b_i)_{d_i}
  x_i^{d_i}}{d_i!},  \\
= & \frac{\Gamma(c)}{\Gamma(a)\Gamma(c-a)}\int_0^1
t^{a-1}(1-t)^{c-a-1}\prod_{i=1}^3(1-x_i t)^{-b_i}\rd t.
\label{eq:FDint} }
The last line analytically continues outside the unit polydisc the
power-series definition of $F_D^{(3)}$. Furthermore, the
continuation to arbitrary parameters $a$ and $c$ is obtainted
through the use of the Pochhammer contour:
\eq{ \int_0^1 \to \frac{1}{(1-\re^{2\pi\ri a})(1-\re^{2\pi\ri
c})}\int_{[C_0,C_1]}. \label{eq:poch} }
Eqs. \cref{eq:FDint} and \cref{eq:poch} can then be used to express
Eq. \cref{eq:Pi} in the form of a sum of generalized hypergeometric
functions. Explicitly, we have
\ea{ \label{eq:Pi4FD} \Pi_4
= & -\re^{\frac{t_0}{z}+\frac{(\a_1-\a_2)t_2}{2}}
\frac{Z_2^{\a_1}Z_3^{\a_2}Z_4^{\a_2}}{Z_1^{\a_2}}  \nn \\
&  \Bigg\{ \frac{\Gamma(-\a_1-\a_2)\Gamma(1+\a_1)}{\Gamma(1-\a_2)}
F_D^{(3)}\l(-\a_1-\a_2,-\a_1,-\a_2,-\a_2,1-\a_2,\frac{Z_1}{Z_2},\frac{Z_1}{Z_3},\frac{Z_1}{Z_4}\r)
\nn \\
& \frac{\Gamma(1-\a_1-\a_2)\Gamma(\a_1)}{\Gamma(1-\a_2)}
F_D^{(3)}\l(1-\a_1-\a_2,-\a_1,-\a_2,-\a_2,1-\a_2,\frac{Z_1}{Z_2},\frac{Z_1}{Z_3},\frac{Z_1}{Z_4}\r)
\nn \\
& \frac{\Gamma(1-\a_1-\a_2)\Gamma(1+\a_1)}{\Gamma(2-\a_2)}
\frac{Z_1}{Z_2}F_D^{(3)}\l(1-\a_1-\a_2,1-\a_1,-\a_2,-\a_2,2-\a_2,\frac{Z_1}{Z_2},\frac{Z_1}{Z_3},\frac{Z_1}{Z_4}\r)\Bigg\},
\\ \nn \\
\label{eq:Pi1FD}
\Pi_1 = & \Pi_4 \quad (Z_1 \leftrightarrow Z_2),\\ \nn \\
\label{eq:Pi2FD} \Pi_2
= & \re^{\frac{t_0}{z}+\frac{(\a_1-\a_2)t_2}{2}}Z_1^{\a_1+\a_2} \nn \\
\Bigg\{ &  \frac{\Gamma(-\a_1-\a_2)\Gamma(\a_1)}{\Gamma(-\a_2)}
F_D^{(3)}\l(-\a_1-\a_2; -\a_1,-\a_2,-\a_2;-\a_2,\frac{Z_2}{Z_1},
\frac{Z_3}{Z_1}, \frac{Z_4}{Z_1}\r) \nn \\
 + & \frac{\Gamma(-\a_1-\a_2)\Gamma(\a_1+1)}{\Gamma(1-\a_2)} F_D^{(3)}\l(-\a_1-\a_2; 1-\a_1,-\a_2,-\a_2;1-\a_2,\frac{Z_2}{Z_1},
\frac{Z_3}{Z_1}, \frac{Z_4}{Z_1}\r) \nn \\
 - & \frac{\Gamma(-\a_1-\a_2)\Gamma(\a_1+1)}{\Gamma(1-\a_2)} F_D^{(3)}\l(-\a_1-\a_2; -\a_1,-\a_2,-\a_2;1-\a_2,\frac{Z_2}{Z_1},
\frac{Z_3}{Z_1}, \frac{Z_4}{Z_1}\r)\Bigg\}
\\ \nn \\
\Pi_3 = & \Pi_2 \quad (Z_1 \leftrightarrow Z_2), \label{eq:Pi4F3} }
where $Z_i(t)$, $i=1,2,3,4$ were defined in Eq. \cref{eq:Zi}.

\subsection{Orbifold point}
\label{sec:B-1}

By Eq. \cref{eq:B-1}, the matrix $\cB$ in Eq. \cref{eq:B-1} is
computed by evaluating the derivatives of $\Pi_i$ at the orbifold
point $x=0$. Consider for simplicity the case $\a=0$. We have
\ea{ Z_1 Z_2^{-1}|_{x=0}= & -1, &  Z_1 Z_3^{-1}|_{x=0}= & \ri, & Z_1
Z_4^{-1}|_{x=0}= & -\ri, \nn \\
Z_2 Z_3^{-1}|_{x=0}= & -\ri, &  Z_2 Z_4^{-1}|_{x=0}= & \ri, & Z_3
Z_4^{-1}|_{x=0}= & -1. }
The value of the Lauricella function, Eq. \cref{eq:FD}, for
arguments equal to distinct roots of unity different from one can be
computed explicitly using the integral representation of Eq.
\cref{eq:FDint}: the symmetry of the Gauss function $~_2
F_1(a,b,c,x)$ under transposition of $a$ and $b$ and simple
manipulations with the products over roots of unity allow to
simplify the integrands down to $t^\b(1-t)^\gamma$ for parameters
$\beta$ and $\gamma$ depending linearly on $\mu_1, \mu_2$. The
integrals are in turn evaluated with the aid of the Euler Beta
integral, Eq. \cref{eq:EulerBeta}. For example, for $i=4$, we have
\ea{ \de_{x_0}\Pi_4|_{x=0} = &  \frac{\re^{\frac{1}{2} \ri \pi
\mu_2}}{\mu_2
  2^{\mu_1+\mu_2+1}} \int_0^1
\frac{(1-q)^{\mu_1-1}(1+q)^{\mu_2+1}}{q^{\frac{\mu_1+\mu_2}{2}}}\frac{\rd
  q}{q} \nn \\
= &
 \frac{\Gamma(\mu_1)\Gamma(-\frac{\mu_1+\mu_2}{2})}{\Gamma(\frac{\mu_1-\mu_2}{2})} \frac{\re^{\frac{1}{2} \ri \pi    \mu_2}}{\mu_2
  2^{\mu_1+\mu_2+1}} ~_2F_1\l( -\frac{\mu_1+\mu_2}{2},-1-\mu_2,
\frac{\mu_1-\mu_2}{2},-1\r) \nn \\
= &
 \frac{\Gamma(\mu_1)\Gamma(-\frac{\mu_1+\mu_2}{2})}{\Gamma(\frac{\mu_1-\mu_2}{2})} \frac{\re^{\frac{1}{2} \ri \pi    \mu_2}}{
  2^{\mu_1+\mu_2+2}}
 \frac{\Gamma(\frac{\mu_1-\mu_2}{2})}{\mu_2\Gamma(-1-\mu_2)\Gamma(1+\frac{\mu_1+\mu_2}{2})}\int_0^1
\frac{(1-q)^{\frac{\mu_1+\mu_2}{2}}}{q^{1/2+\mu_2/2}}\frac{\rd
q}{q} \nn \\
= & \frac{\re^{\frac{1}{2}\pi \ri
\mu_2}}{4}\frac{\Gamma\l(\frac{\mu_1}{2}\r)\Gamma\l(-\frac{\mu_1+\mu_2}{2}\r)}{\Gamma(1-\mu_2)}
}
The value of the derivatives with respect to $x_\a$ for $\a>0$ are
computed in the same way; the final result is Eqs. \cref{eq:Bfin}--\cref{eq:I}.

\subsection{Large radius}
\label{sec:A}

By the discussion at the end of the proof of Proposition
\ref{prop:AB}, twisted periods behave around large radius as
\eq{
 \Pi_i(t,\a) \sim  z\l(\cA^{-1}_{i,1}+\cA^{-1}_{i2} \re^{-t_1 \mu_1}+\cA^{-1}_{i3}\re^{-t_2 \mu_2}+\cA^{-1}_{i,4}\re^{t_3 (\mu_1+\mu_2)}\r).
}
When $\Re(t)\to -\infty$,
the arguments of the Lauricella functions appearing in the
expression of $\Pi_i$ behave like
%
\ea{ \label{eq:ZLR1}
(Z_2 Z_1^{-1},Z_2 Z_3^{-1},Z_2 Z_4^{-1}) & \sim (-\infty,\infty,\infty), \\
\label{eq:ZLR2}
(Z_2 Z_1^{-1},Z_3 Z_1^{-1},Z_4 Z_1^{-1}) & \sim (-\infty,0,1), \\
\label{eq:ZLR3}
(Z_1 Z_2^{-1},Z_3 Z_2^{-1},Z_4 Z_2^{-1}) & \sim (0,0,0), \\
\label{eq:ZLR4} (Z_1 Z_2^{-1},Z_1 Z_3^{-1},Z_1 Z_4^{-1}) & \sim
(0,-\infty,1). }
The simplest asymptotics is for $i=3$, as it is dictated by the
convergent power series expansion of Eq. \cref{eq:FD}:
\ea{ \Pi_3 \sim &  \re^{\frac{t_0}{z}+\frac{(\mu_1-\mu_2)t_2}{2}}
Z_1^{\mu_1+\mu_2}
\frac{\Gamma(-\mu_1-\mu_2)\Gamma(\mu_1-1)}{\Gamma(-\mu_2)} \nn \\
\sim & -\re^{\frac{t_0}{z}-\mu_2
t_2}\frac{\Gamma(-\mu_1-\mu_2)\Gamma(1+\mu_2)}{\Gamma(1-\mu_1)}. }
This sets $\cA_{3,j}=\delta_{j,3}
\frac{\Gamma(-\mu_1-\mu_2)\Gamma(1+\mu_2)}{\Gamma(1-\mu_1)}$. \\

The other cases are more delicate. One strategy to treat them, as in
\cite{Brini:2013zsa}, is to resum Eq. \cref{eq:FD} in one of the
variables and then apply the Kummer formulas to the summand, which
in all cases has the form of a Gauss function in the resummed
variable. For $\Pi_2$ and $\Pi_4$, we use that, when
$(x_1,x_2,x_3)\sim (0,\infty,1)$,
$F_D^{(3)}(a,b_1,b_2,b_3,c,x_1,x_2,x_3)  \sim
F_1(a,b_2,b_3,c,x_2,x_3)$, where
\eq{ F_1(a,b_2,b_3,c,x_2,x_3)=\sum_{m,n\geq
  0}\frac{(a)_{m+n}(b_2)_m(b_3)_n}{(c)_{m+n}m!n!}x_2^m x_3^n
\label{eq:F1} }
is the Appell $F_1$ function. Performing the summation on $n$ for
fixed $m$ in Eq. \cref{eq:F1} gives
\ea{ & F_1(a,b_2,b_3,c,x_2,x_3) = \frac{\Gamma(c)}{\Gamma(a)}
\sum_{m\geq
  0}\frac{x_2^m (b_2)_m\Gamma(a+m)}{m!\Gamma(c+m)} ~_2F_1(a+m,b_3,c+m,x_3)
\nn \\ & \frac{\Gamma (c) \left(1-x_3\right)^{-a-b_3+c} \Gamma
\left(a-c+b_3\right) }{\Gamma (a) \Gamma \left(b_3\right)} \sum
_{k=0}^{\infty } \frac{x_2^k \left(b_2\right)_k \,
   _2F_1\left(c-a,c+k-b_3;-a+c-b_3+1;1-x_3\right)}{k!} \nn \\ & +\frac{\Gamma (c) \Gamma \left(-a+c-b_3\right)}{\Gamma (c-a)} \sum
   _{k=0}^{\infty } \, _2F_1\left(a+k,b_3;a-c+b_3+1;1-x_3\right) \frac{x_2^k
  (a)_k \left(b_2\right)_k}{k! \Gamma \left(c+k-b_3\right)}
}
The leading asymptotics at $x_3 \sim 1$ is therefore given by
\ea{
& F_1(a,b_2,b_3,c,x_2,x_3) \nn \\
& \sim \frac{\Gamma (c)  \Gamma \left(a-c+b_3\right) }{\Gamma (a)
\Gamma
\left(b_3\right)}\left(1-x_3\right)^{c-a-b_3}\left(1-x_2\right)^{-b_2}
+\frac{\Gamma (c) \Gamma
  \left(-a+c-b_3\right)}{\Gamma (c-a)\Gamma(c-b_3)}
\, _2F_1\left(a,b_2;c-b_3;x_2\right), }
and further application of the Kummer formula at infinity on $x_2$
yields \ea{ F_1(a,b_2,b_3,c,x_2,x_3)  \sim &
\frac{\Gamma (c)  \Gamma \left(a-c+b_3\right) }{\Gamma (a) \Gamma
\left(b_3\right)}\left(1-x_3\right)^{c-a-b_3}\left(-x_2\right)^{-b_2}
+\frac{\Gamma (c)}{\Gamma (c-a)}
  \frac{\Gamma(b_2-a)}{\Gamma(b_2)}(-x_2)^{-a} \nn \\ & + \frac{\Gamma (c) \Gamma
  \left(c-a-b_3\right)\Gamma(a-b_2)}{\Gamma (c-a)\Gamma(a)\Gamma(c-b_3-b_2)}(-x_2)^{-b_2}.
}
Hence:
\ea{ \re^{-\frac{t_0}{z}}\Pi_4 \sim & \frac{\Gamma (\mu_1) \Gamma
(-\mu_1-\mu_2) \re^{-\mu_1 t_1}}{\Gamma
  (1-\mu_2)} -\frac{\Gamma (-\mu_1) \Gamma (\mu_1+\mu_2)}{\Gamma
  (1+\mu_2)}+\frac{\Gamma (\mu_1) \Gamma (-\mu_1-\mu_2) \re^{(\mu_1+\mu_2) t_3}}{\Gamma(1-\mu_2)}, \\
\re^{-\frac{t_0}{z}}\Pi_2 \sim & -\frac{\re^{\ri \pi  \left(\mu
_1+\mu _2\right)} \Gamma \left(-\mu _2\right) \Gamma \left(\mu
_1+\mu _2\right)}{\Gamma \left(1+\mu
   _1\right)}+\frac{\re^{\ri \pi  \left(\mu _1+\mu _2\right)} \Gamma
  \left(-\mu _1-\mu _2\right) \Gamma \left(\mu _2\right) \re^{-\mu
    _2 t_2}}{\Gamma
   \left(1-\mu _1\right)}\nn \\ - & \frac{
\Gamma \left(\mu _1\right) \Gamma \left(-\mu _1-\mu _2\right)
   \re^{(\mu _1+\mu _2)t_3}}{\Gamma \left(1-\mu _2\right)}.
}
Finally, for $\Pi_1$ we use that
\ea{ F_D^{(3)}(a; b_1,b_2,b_3 ;c;x_1,x_2,x_3) = &
(-x_2)^{-b_2}F_1\l(a-b_2,
  b_1,b_3,c-b_2, x_1,x_3\r)\l(1+\cO\l(\frac{1}{x_2}\r)\r)
  \nn \\ + & (-x_2)^{-a}
\frac{\Gamma(c)\Gamma(b_2-a)}{\Gamma(b_2)\Gamma(c-a)}\l(1+\cO\l(\frac{1}{x_2}\r)\r)
}
where we have resummed w.r.t. $x_2$, applied Lemma \ref{lem:kum} for
$q=1$, and isolated the leading contribution in $x_2$ for $x_1/x_2
\sim 0$, $x_3/x_2 \sim 0$, as is the case when $\mathfrak{Re}(t)\sim
-\infty$ by Eqs. \cref{eq:ZLR1}--\cref{eq:ZLR4}. Setting now $x_1=x_3$ and further application of
Lemma \ref{lem:kum} gives
\ea{ & F_D^{(3)}(a; b_1,b_2,b_3 ;c;x_1,x_2,x_3) \sim (-x_2)^{-a}
\frac{\Gamma(c)\Gamma(b_2-a)}{\Gamma(b_2)\Gamma(c-a)}
 \nn \\
+ &  (-x_2)^{-b_2}
\frac{\Gamma(c)\Gamma(a-b_2)}{\Gamma(a)\Gamma(c-b_2)}
{}_2F_1\l(a-b_2, b_1+b_3,c-b_2, x_1\r) \nn \\
& \sim (-x_2)^{-a}
\frac{\Gamma(c)\Gamma(b_2-a)}{\Gamma(b_2)\Gamma(c-a)}
 \nn \\
+ &  (-x_2)^{-b_2}(-x_1)^{b_2-a}
\frac{\Gamma(c)\Gamma(a-b_2)\Gamma(b_1+b_3+b_2-a)}{\Gamma(b_1+b_3)\Gamma(c-a)\Gamma(a)} \nn \\
+ &  (-x_2)^{-b_2}(-x_1)^{-b_1-b_3}
\frac{\Gamma(c)\Gamma(a-b_2-b_1-b_3)}{\Gamma(a)\Gamma(c-b_1-b_2-b_3)},
}
so that
\ea{ \re^{-\frac{t_0}{z}}\Pi_1 \sim & \frac{(-1)^{\mu _1} \Gamma
\left(\mu
   _1\right) \Gamma \left(-\mu _1-\mu _2\right) \re^{-\mu _1t_1}}{\Gamma \left(1-\mu _2\right)}-\frac{\Gamma \left(-\mu _1\right) \Gamma \left(-\mu _2\right)}{\Gamma \left(1-\mu _1-\mu _2\right)}-\frac{\Gamma \left(\mu _1\right) \Gamma \left(\mu
   _2\right) \re^{-\mu _2 t_2}}{\Gamma \left(1+\mu _1+\mu _2\right)},
}
which concludes the proof.

\providecommand{\bysame}{\leavevmode\hbox to3em{\hrulefill}\thinspace}
%
%

\bibliographystyle{amsalpha}
\bibliographymark{References}
\def\cprime{$'$}

\end{document}